\documentclass [10pt] {amsart}

\usepackage[latin1]{inputenc}%
\usepackage{amssymb,amsfonts,epsfig,amsmath,graphics,graphicx}%
\usepackage{amsthm}%
\usepackage[all]{xy}%
\usepackage{verbatim}%
\usepackage{enumitem}
\usepackage{bm}

\title {Dynamics on trees of spheres.}

\newtheorem{theorem} {Theorem}[section]
\newtheorem{proposition}[theorem]{Proposition}
\newtheorem{lemma}[theorem]{Lemma}
\newtheorem{corollary}[theorem]{Corollary}
\newtheorem{definition}[theorem]{Definition}
\newtheorem{question}[theorem]{Question}

\newtheorem{example}[theorem]{Example}

\theoremstyle{plain}
\newtheorem{theoremint} {Theorem}

\newtheorem*{definitionint}{Definition}

\newtheorem*{lemmaint}{Claim}

\theoremstyle{remark}
\newtheorem{remark}[theorem]{Remark}

\renewenvironment{proof}{\noindent{\bf Proof. }}{\hfill{$\square$} \vskip.3cm}
\newenvironment{recall}{\noindent{\bf Recall. }}{\medskip}

\def\cal{\mathcal}

\def\C{{\mathbb C}}

\def\F{{\mathcal F}}

\def\N{{\mathbb N}}

\def\S{{\mathbb S}}
\def\St{{\mathcal S}}
\def\T{{\mathcal T}}

\def\Z{{\mathbb Z}}

\def\ts{{tree of spheres }}
\def\tss{{trees of spheres}}
\def\tmb{{tree marked by}}

\def\Aut{{\rm Aut}}
\def\Rat{{\rm Rat}}
\def\rat{{\rm rat}}

 \def\epsilon{{\varepsilon}}
\def\val{{\rm val}}
\def\card{{\rm card}}
\def\deg{{\rm deg}}
\def\mult{{\rm mult\,}}
\def\Crit{{\rm Crit}}
\def\Prep{{\rm Prep}}

\keywords{holomorphic dynamics, rescaling limits, trees, compactification}
\address{Pontificia Universidad Cat\'olica de Valparaíso\\
Instituto de Matem\'aticas\\ 
Av. Brasil 2950\\
Valparaíso (Chile)}
\email{matthieu.arfeux@pucv.cl}

\usepackage{hyperref}

\hypersetup{
backref=true, 
pagebackref=true,
hyperindex=true, 
colorlinks=true, 
breaklinks=true, 
urlcolor= blue, 
linkcolor= blue, 
}

\begin{document}

\title{Dynamics on Trees of Spheres}

\author{Matthieu Arfeux}

\maketitle

\begin{abstract}

We introduce the notion of dynamically marked rational map and study sequences of analytic conjugacy classes of rational maps which diverge in the moduli space. In particular, we are interested in the notion of rescaling limits introduced by Jan Kiwi. For this purpose, we introduce the notions of trees of  spheres, covers between them and dynamical covers between them. 
We will study fundamental properties of these objects. We prove that they appear naturally as limits of marked spheres, respectively marked rational maps and dynamically marked rational maps.

 We also prove that a periodic sphere in a dynamical cover between trees of spheres corresponds to a rescaling limit. We  recover as a byproduct a result of Jan Kiwi regarding the bound on the number of rescaling limits that are not post-critically finite. 

\end{abstract}


\section{Introduction}

The goal of this article is to introduce and study a class of dynamical systems on trees of
spheres. First, we examine the trees of spheres that arise from a special case of the Deligne-Mumford compactification of the moduli space of stable curves, in genus 0 with marked
points \cite{DM}. (See Figure \ref{ex1} and Definition \ref{TOFdef}; see also \cite{FT} and \cite{A1}.) Then, we introduce
covers between trees of spheres; the novelty here is that we impose certain restrictions on
the maps (Definition 	\ref{Defportrait}) that allow for interesting properties such as the existence of
a global degree and a Riemann-Hurwitz type formula. Finally, we introduce dynamical
systems between trees of spheres. These are covers between trees of spheres
$$\F:\T^Y\to\T^Z,$$
together with the data of a compatible subtree $T^X$ whose vertices lie in both $T^Y$ and $T^Z$ (see
Definition \ref{compdef}). The definition allows for the iteration of the map $\F$. The collection of such
systems includes rational maps on the Riemann sphere 
$$f:\S\to\S$$
with a finite collection of marked points. We are interested in applying this theory to the
study of rational maps and their moduli spaces.

The objects just mentioned also exist in a topological category. Our first result holds in
this more flexible setting:

\begin{theoremint} \label{thmkiw1}A dynamical system of topological trees of spheres of degree $D$ has at most $2D-2$ cycles of spheres which are not post-critically finite. 
\end{theoremint}

Recall that a rational map is post-critically finite if each of its critical points has finite forward orbit; the definition is similar for a periodic cycle of spheres within a tree-of-spheres
dynamical system (see Definition \ref{deftriviall}).

For applications, we require the topology on the space of trees of spheres and
dynamical systems between trees of spheres as introduced in \cite{A} and \cite{A1}. For this, from Section \ref{chap4}, we consider all objects in an analytic category. In this article we make precise three notions of convergence:

\begin{enumerate}
\item the convergence of a sequence of marked spheres to a tree of spheres.     
\item the convergence of a sequence of marked spheres covers to cover between trees of spheres.
\item the convergence of a sequence of dynamical systems of marked spheres to a dynamical system between trees of spheres. 
\end{enumerate}
The associated topology is studied in detail in \cite{A1} and is not Hausdorff. 

A sequence of rational maps $f_n : \S \to\S$ can converge to a dynamical system on a
tree of spheres. Periodic cycles of the limiting system correspond to the ``rescaling limits"
introduced by Jan Kiwi \cite{Kiwi2}, defined as:

\begin{definitionint}
For a sequence of rational maps $(f_n)_n$ of a given degree, a {rescaling} is a sequence of Moebius transformations $(M_n)_n$ such that there exist $k\in\N$ and a rational map $g$ of degree $\geq 2$ such that
\[M_n\circ f_n^k\circ M_n^{-1}\to g\]
uniformly in compact subsets of $\S$ with finitely many points removed. Such a $g$ is called a rescaling limit.

\end{definitionint}

Jan Kiwi introduced a notion of dynamical dependence among rescalings (see Section
\ref{indepresclim}). We prove in Theorem \ref{lienkiwi0} that disjoint periodic cycles of spheres for a dynamical system between trees of spheres, which is a limit of rational maps, correspond to a dynamically independent rescalings. We prove the converse in Theorem \ref{intro3}, and we obtain a new
proof of Kiwi's result:

\begin{theoremint}\label{alpha} \cite{Kiwi2} For every sequence in $\Rat_d$ for $d\geq 2$, there are at most $2d-2$ classes of dynamically independent rescalings with a non post-critically finite rescaling limit. 
\end{theoremint}
Kiwi's proof uses the analysis and dynamics of rational maps on a Berkovich space, associated to the completion of the field of formal Puiseux series in a variable t with complex coefficients, equipped with
the non-archimedean absolute value measuring the order of vanishing at $t = 0$.

The objects introduced in this paper have appeared in various forms in the literature.
See, for example:
\begin{enumerate}
\item in \cite{Arno}, \cite{HK}, \cite{Ko} and \cite{S} in the context of application of Thurston's results concerning the characterization of post-critically finite topological branched covers that are realizable as rational maps and the study of related Teichmüller spaces;
\item in  \cite{Treesph}, \cite{S1} and \cite{S2} where the authors encode dynamical systems by means of suitably associated trees;
\item in the use of Berkovich spaces in the context of holomorphic dynamics such as  \cite{CubPol2},
\cite{XetDem} and \cite{Kiwi2}.
\end{enumerate}

The list is not exhaustive. Our main goal in this article is to provide a systematic study of
these trees of spheres and maps between them. Article \cite{A2} continues the study of rescaling
limits; article \cite{A1} is focused on the isomorphism classes of the objects introduced here and
the natural topology on the associated spaces. These three articles were developed from
the results of \cite{A}.

\medskip
\paragraph{\bf Outline} In Section \ref{chap2}, we focus on the description of non dynamical objects. We recall some vocabulary about trees, introduce covers and prove in particular the Riemann-Hurwitz formula (Proposition \ref{RH}).
In Section \ref{chap3}, we introduce dynamics and prove Theorem \ref{npcf} which is a stronger version of Theorem \ref{thmkiw1}.

In Section \ref{chap4}, we define the notions of marked spheres, cover between marked spheres and dynamical systems between trees of spheres. We define the convergence notion cited above and prove some technical lemmas about them. In Section \ref{chap5}, we recall the definitions concerning rescaling limits from \cite{Kiwi2} and prove Theorem \ref{lienkiwi0} and Theorem \ref{intro3}. In the proofs of these theorems, we require
two compactification results from \cite{A1} and \cite{FT}; these are recalled as Theorem \ref{thmcomp} and
Theorem \ref{thmcomp00}.

\medskip
\paragraph{\bf Acknowledgments.}
I would like to thank my PhD advisor Xavier Buff for all the time he spent with me in order to transform an idea into a paper. I also want to thank the referee for all the time spent on this paper, and help offered.

\section{Non dynamical objects}\label{chap2}

\subsection{Combinatorial trees}


Recall that a (simple undirected) graph is the disjoint union of a finite set $V$, called the set of vertices, and another finite set $E$ consisting of elements of the form $\{v,v'\}$ with distinct $v,v'\in V$ called the set of edges.
We say that $\{ v,v'\}$ is an edge between $v$ and $v'$. For all $v\in V$ we define $E_v$ to be the set of edges containing $v$.
The cardinal of $E_v$ is called {\it valence} of $v$ and denoted by $\val(v)$.

In a graph $T$, a path is a one-to-one map $t:[1,k]\to T$ such that for $j\in [1,k-1]$,
\begin{enumerate}
\item if $t(j)$ is a vertex, then $t(j+1)$ is an edge and $t(j+1)\in E_{t(i)}$ and 
\item if $t(j)$ is an edge, then $t(j+1)$ is a vertex and $t(j)\in E_{t(j+1)}$. 
\end{enumerate}
 We say that this path connects $t(1)$ to $t(k)$. We will identify a path and its image. 
 We say that a path is connected if each vertex is connected to any other distinct one.

For a graph $T$, a cycle is a one-to-one map $t:\Z/k\Z\to T$ such that for  $j\in\Z/k\Z$, the following holds:
\begin{enumerate}
\item if $t(j)$ is a vertex, then $t(j+1)$ is an edge and $t(j+1)\in E_{t(i)}$ and 
\item if $t(j)$ is an edge, then $t(j+1)$ is a vertex and $t(j)\in E_{t(j+1)}$. 
\end{enumerate}

\begin{definition}[Tree] A tree is a connected graph without any cycle. 
\end{definition}

See for example Figure \ref{figure1}. If a graph has no cycles then it is well known that there is always a unique path connecting two distinct vertices (see for example \cite[Theorem 1.5.1]{graphtheory}). 
For a tree $T$ we will denote by $[v_1,v_2]$ the unique path of $T$ connecting $v_1$ to $v_2$. 
 
 The path $t$ will be denoted sometimes by $[t(1),t(3),t(5),\ldots,t(k))]$ if $t(1)$ and $t(k)$ are vertices, or $]t(2),t(4),t(6),\ldots,t(k-1))[$ if $t(1)$ and $t(k)$ are edges.

A connected sub-graph of a tree $T$ is a connected graph without cycles. This is also a tree and we say that it is a sub-tree of $T$.

In a tree, vertices with valence $1$ are called leaves. The other ones are called internal vertices. We denote by $IV$ the set of Internal Vertices.

 \begin{figure}
\centerline{\includegraphics[width=8cm]{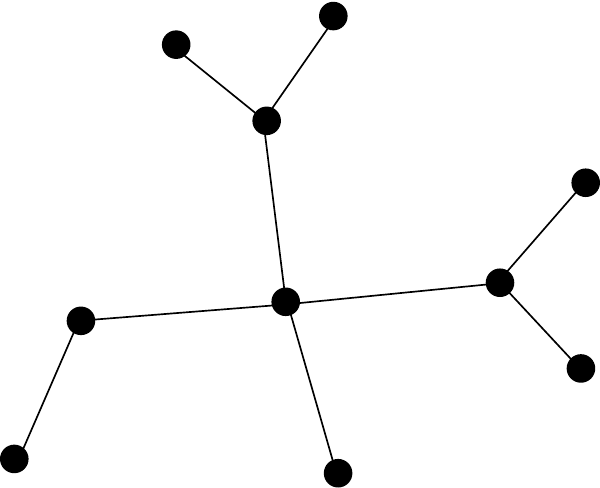}}
 \caption{On this example, vertices are represented as dots and we represent by a segment an edge connecting the corresponding vertices. }
\label{figure1} \end{figure}

 

A graph $T$ has a natural topology such that the closed sets are unions of sub-trees. 

\begin{definition}[Connected component]
The connected component of a sub-graph $T'\subset T$ is the connected sub-graph of $T$ that is maximal for the inclusion. 
\end{definition}

\begin{definition}[Branch]
For $v$ a vertex of a tree $T$ and for $\star\in T-\{v\}$, a branch of $\star$ on $v$ is the connected component of $T-\{v\}$ containing $\star$. It is denoted by $B_v(\star)$.
\end{definition}

Let $v\in V$. As $T$ is a tree, for all $\star\in T-\{v\}$, there is a unique path connecting $v$ to $\star$. By definition this path contains a unique edge $e\in E_v$ so each branch on $v$ will be denoted $B_v(e)$ with $e\in E_v$.

 \begin{figure}
  \centerline{\includegraphics[width=8cm]{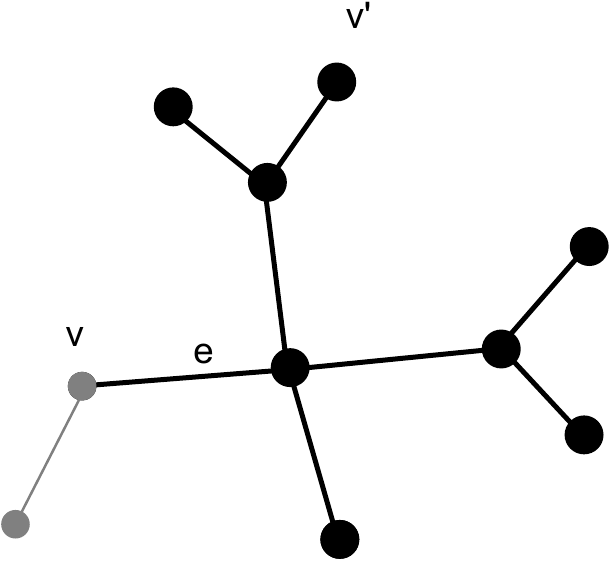}} 
   \caption{ On this example the branch $B_v(e)$ (or $B_v(v')$) is colored dark.}
\label{ex0} \end{figure}



In the following, we introduce a tool called characteristic which is similar to the Euler characteristic and will be useful when we talk as well about covers between trees of spheres. We will have a Riemann-Hurwitz formula.

\begin{definition}[Characteristic of a sub-graph]\label{defcar}
The characteristic of a vertex $v$ of a graph $T$ is \[\chi_T(v):=2-\val(v).\]
The characteristic of a sub-graph $T'$ of $T$ is the integer 
\[\chi_T(T'):=\sum_{v\in V\cap T'} \chi_T(v).\]
\end{definition}

We will simply use the notation $\chi(T')$ when it is not confusing.
Cf Figure \ref{That} for an example.

\begin{lemma}\label{val}
For any tree $T$, we have $\chi_T(T)=2$. 
\end{lemma}

\begin{proof}
Observe first that on a graph, each vertex $v$ is connected to $\val(v)$ edges and that each edge is connected to two vertices. Then we have
\[\sum_{v\in V} \val(v) = 2\card(E).\]
Moreover, in a tree, we have $\card{V}=\card{E}+1$ (see \cite[corollary 1.5.3]{graphtheory} for example). It follows that   
\[\chi_T(T)=\sum_{v\in V} \bigl(2-\val(v) \bigr)= 2\card{V} - 2\card{E} = 2.\]
 \end{proof}

Recall that the closure of a set is the smallest closed set containing it (cf Figure \ref{That}).

\begin{definition}
If $T'\subseteq T$, we denote by 
\begin{enumerate}
\item $\overline T'$ the closure of $T'$  in $T$ and 
\item $\partial_T T':=\overline T'-T'$ the boundary of $T'$ in $T$.
\end{enumerate}
\end{definition}

\begin{lemma}\label{sarbre}
If $T'$ is open an connected in $T$, then the boundary $\partial_T T'$ is the set of vertices $v\in T-T'$ lying to an edge of $T'$. 
The closure $\overline T'$ is a sub-tree of $T$ for which the set of internal vertices is $IV\cap T'$.
\end{lemma}

\begin{proof}
The closure of $T'$ is the smallest sub-graph of $T$ containing $T'$. 
It has to contain all vertices  $v\in T$ lying on an edge of $T'$. It is not necessary to add other vertices or other edges in order to obtain a graph. This proves that $\partial_T T'$ is the set of vertices $v\in T-T'$ lying on an edge of $T'$.

The closure of a connected open set is a sub-graph of $T$. So it is a sub-tree of $T$. The vertices of $\partial_T T'$ are the leaves of $\overline T'$. If it is not the case then $T'=\overline T'-\partial_T T'$ would not be connected. 
The set $T'$ is open, so for all vertex $v$ of $T'$, we have $E_v\subset T'$. Consequently the valence of $v$ in $\overline T'$ is the same as the one of $v$ in $T$. 
This proves that internal vertices of $\overline T'$ are internal vertices of $T$ contained in $T'$. 
 \end{proof}

  \begin{figure}
  \centerline{\includegraphics[width=12cm]{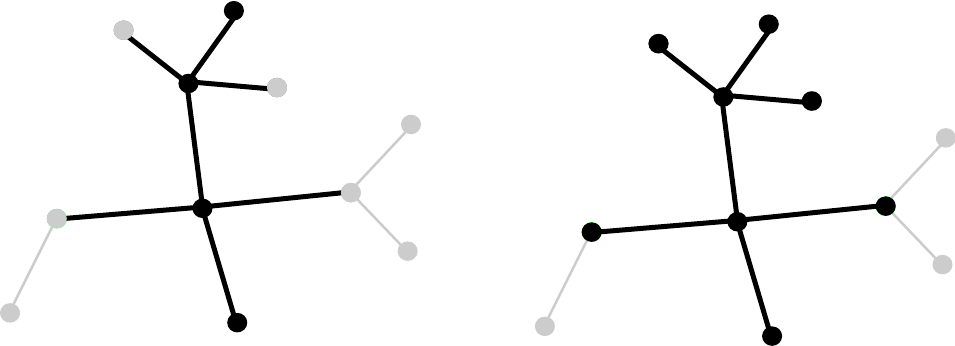}}
   \caption{ An open connected sub-graph of $T'$ in dark of characteristic $-2$ on the left and $\overline T'$ in dark on the right.}\label{That}
 \end{figure}

\begin{lemma}\label{carautre}
If $T'$ is a non empty sub-graph of $T$, open and connected, then 
\[\chi_T(T')=2-\card \partial_T T'.\]
\end{lemma}

\begin{proof} 
In $\overline T'$, each vertex $v$ of $T'$ has valence $\val(v)$ and each vertex of $\partial_T T'$ has characteristic $1$. According to Lemma \ref{val}, we conclude that
\[2 = \chi_T(\overline T') = \sum_{v\in V\cap T'}\chi_T(v) + \sum_{v\in \partial_T T'} \chi_T(v) = \chi_T(T') + \card \partial_T{T'}.\]
 \end{proof}

\begin{lemma}\label{branch}
If $T'$ is a non empty sub-graph of $T$, open and connected, then 
\begin{enumerate}
\item $\chi_T(T')\leq 2$ ;
\item $\chi_T(T')=2$ iff $T'=T$ ;
\item $\chi_T(T')=1$ iff $T'$ is a branch of $T$. 
\end{enumerate}
\end{lemma}

\begin{proof}
According to the previous lemma, $\chi_T(T') = 2-\card\partial_T{T'}$.

1) clear.  

2) $\chi_T(T')=2$ iff $\partial_T T'=\emptyset$ iff $T'$ is open and closed iff $T'=T$. 

3) If $T'$ is a branch, we have a vertex $v$, then $\partial_T T'=\{v\}$ and $\chi_T(T')=1$. Conversely,  if $\chi_T(T')=1$, then $\partial_T T'$ contains a unique vertex $v$. Let $e:=\{v,v'\}$ be the edge of $T'$ containing $v$ and define $B:=B_v(e)$. As $T'$ is connected, contained in $T-\{v\}$ and contains $e$, we have $T'\subseteq B$. Given that $T'\cap B=\emptyset$, the branch $B$ is the disjoint union of two open sets $T'$ and $B-\overline T' = B-T'$. As $B$ is connected, we have $B-T'=\emptyset$ and it follows that $B=T'$. 

 \end{proof}


\newpage

\subsection{Combinatorial trees maps}

\begin{definition}[Trees map]   A map $F:T\to T'$ is a trees map if
\begin{enumerate}
\item $T$ and $T'$ are trees;
\item vertices map to vertices : $F(V)\subseteq V'$;
\item every edge connecting two vertices maps to an edge connecting the image of these vertices : if $\{v,w\}\in E$, then $F(\{v,w\}) = \{F(v),F(w)\}\in E'$.
\end{enumerate} 
\end{definition}

For an example of such a map see Figure \ref{ex1}.

Observe that if $U$ is a sub-graph of $T$, then $F(U)$ is a sub-graph of $T'$ and, conversely, if $U'$ is a sub-graph of $T'$, then $F^{-1}(U')$ is a sub-graph of $T$. Particularly, the preimage of closed sets are closed:

 \begin{figure}
 \centerline{\includegraphics[width=10.15cm]{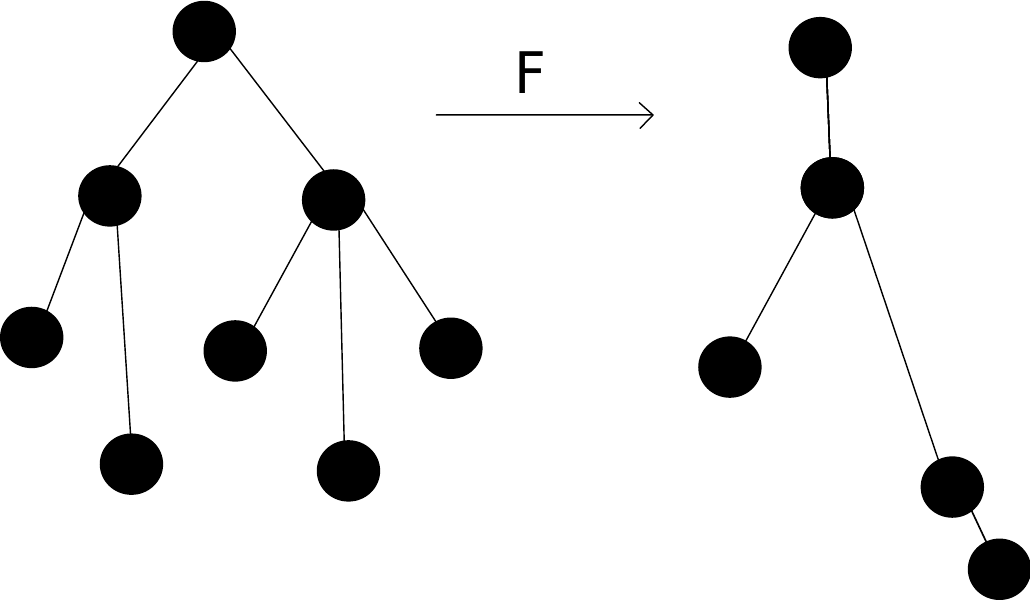}}
   \caption{Example of a trees map: the image of a vertex is the vertex at the same horizontal level. In this example the map is not surjective.}\label{ex1} \end{figure}

\begin{proposition} \label{imageconnexe}
Trees maps are continuous and the image of a sub-tree is a sub-tree.
\end{proposition}

\begin{proof} 
A connected set maps to a connected one. \end{proof}


\subsection{Trees of spheres}

From now on, $X$, $Y$ and $Z$ denote finite sets with at least 3 elements.

\begin{definition}[Marked tree]
A tree $T$ marked by $X$ is a tree such that the leaves are the elements of $X$.
 \end{definition}

A \tmb\,$X$ will be denoted by $T^X$. Later in this paper we will consider many trees that are marked by different sets. When an object $O$ is associated with a tree marked by a set $X'$, we will denote it by $O^{X'}$ to help the reading. 
For example, the set of edges of $T^X$ will be denoted by $E^X$ instead of simply $E$.

\begin{definition} [Marked tree of spheres]\label{TOFdef}
A (topological) \ts ${\cal T^X}$ (marked by $X$) is  the data of: 
\begin{enumerate}
\item a combinatorial tree $T^X$ and 
\item for every internal vertex $v$ of $T^X$, 
\begin{enumerate}
    \item a $2$-dimensional topological sphere $\St_v$ and
    \item a one-to-one map $i_v:E_v\to \St_v$. 
    \end{enumerate}
\end{enumerate}
\end{definition}

For $e\in E_v$, we say that $i_v(e)$ is the {\it attaching point} of $e$ on $v$. We will often use the notation $e_v := i_v(e)$ sometimes even $i_v(v'):=e_v$ if $v'\in B_v(e)$.
We define $X_v : =i_v(E_v)$ to be the set of attaching points on the sphere ${\cal S}_v$.

 \begin{remark} Giving a one-to-one map $i_v:E_v\to S_v$, is the same as giving a map
  $a_v:X\to S_v$ such that $a_v(x_1)=a_v(x_2)$ if and only if $x_1$ and $x_2$ are in the same corresponding branch of $v$. This means $a_v(x) := i_v(e)$ if $x$ lies in $B_v(e)$. 
   \end{remark}
 
\begin{example}[Marked spheres]
A \ts marked by $X$ with a unique internal vertex $v$ is the same data as this vertex $v$ and the map $i_v$. We call it a marked sphere or a sphere marked by $X$.
\end{example} 

 \begin{figure}
   \centerline{\includegraphics[width=7.15cm]{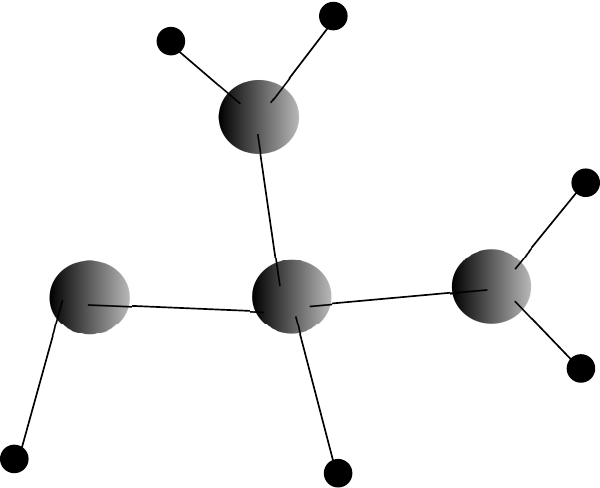}}
   \caption{ A \ts having the same combinatorial tree as the one on Figure \ref{ex0}.}
 \label{Fig5}\end{figure}


\subsection{Covers between trees of spheres}

\paragraph{Definitions and degree}
A cover between trees of spheres is the extension to trees of spheres of the notion of combinatorial trees cover. We add the data of a branched cover for each internal vertex and require that the branching locus is contained in the set of the edges attaching points.

\begin{definition} [Cover]
A cover between trees of spheres  ${\cal F}:{\cal T}^Y\to {\cal T}^Z$ is the following data : 
\begin{enumerate}
\item a trees map $F:T^Y\to T^Z$ mapping leaves to leaves and internal vertices to internal vertices ( $F(Y)\subseteq Z$ and $F(IV^Y)\subseteq IV^Z$ ) and
\item for every internal vertex $v\in IV^Y$ and $w:=F(v)\in IV^Z$, 
a topological branched covering $f_v: \St_v\to \St_w$ such that
\begin{enumerate}
\item the restriction $f_v: \St_v-Y_v\to \St_w-Z_w$ is a cover;
\item $f_v\circ i_v=i_w\circ F$ on $E_v$;
\item if $e=\{v_1,v_2\}\in E^Y$ is an edge connecting two internal vertices, 
then ${\rm deg}_{e_{v_1}} f_{v_1} = {\rm deg}_{e_{v_2}} f_{v_2}$.  
\end{enumerate}
\end{enumerate}
\end{definition}


For every internal vertex $v\in IV^Y$, we denote $\deg(v):=\deg(f_v)$ in order to simplify the expression. 
As well, for all $x\in \St_v$ we define $\deg(x):=\deg_xf_v$.
 The condition 3 assures that we can define a degree for every edge $e$ connecting two internal vertices $v_1$ and $v_2$ of $T^Y$, that will be denoted by \[\deg(e):=\deg_{e_{v_1}} f_{v_1} = \deg_{e_{v_2}} f_{v_2}.\]

Each leaf $y\in Y$ is connected to a unique internal vertex $v$ by an edge $e$, so we can define  \[\deg(y):=\deg(e):=\deg_{e_v} f_v.\]
This define a degree map $\deg_{F}|_Y$ for the map $F|_Y:Y\to Z$. We will see that $(F|_Y,\deg_{F}|_Y)$ is a portrait in the following sense.

\begin{definition}\label{Defportrait}
A portrait ${\bf F}$ of degree $d\geq 2$ is a pair $(F,\deg)$ where 
\begin{enumerate}
\item $F:Y\to Z$ is a map between finite sets $Y$ and $Z$ and
\item $\deg:Y\to \N-\{0\}$ is a map that verifies
\[\sum_{a\in Y}\bigl(\deg(a) -1\bigr) = 2d-2\quad\text{and}\quad \sum_{a\in F^{-1}(b)} \deg(a) = d\quad\text{ for all } b\in Z.\] 
\end{enumerate}
\end{definition}

\begin{definition}
A critical vertex (resp. critical leaf) of $\F$ is a vertex of $T^Y$ (resp. a leaf $y\in Y$) having degree more than one.
We then define $\mult (y):=\deg (y)-1$ to be the multiplicity of $y$. We denote by $\Crit \F$ the set of critical leaves of $\F$. 
 \end{definition}

For each vertex $v$ of $T^Z$ and each leave $e$ of $T^Z$ we can define
\[D_v:=\sum_{v'\in F^{-1}(v)} {\rm deg}(v')\quad\text{and}\quad 
D_e:=\sum_{e'\in F^{-1}(e)} {\rm deg}(e').
\]

\begin{lemma} If $e\in E_v$, then $D_e=D_v$. \label{defdeg}
\end{lemma}

\begin{proof} If $v$ is a leaf, then preimages  $v'$ of $v$ are leaves on which there are attached preimages $e'$ of $e$. The lemma is clear because by definition we have ${\rm deg}(v')= {\rm deg}(e')$. 

When $v$ is an internal vertex, let $X$ be the set of points $x$ lying in the sphere $\St_{v'}$ with $F(v')=v$ and $f_{v'}(x) = e_v$. 
Let $x\in X$. Given that $e_v\in Z_v$ and that ${f_{v'}:\St_{v'}-Y_{v'}\to \St_v-Z_v}$ is a cover, then $x\in Y_{v'}$. Consequently, $x$ is the attaching point of an edge $e'$ of $T^Y$ mapped to $e$. Inversely, if $F(e')=e$, then $e'$ is attached to a sphere $v'\in F^{-1}(v)$ at a point $x\in X$. So we have
\begin{align*}
D_e = \sum_{e'\in F^{-1}(e)} {\rm deg}(e') &= \sum_{x\in X}  {\rm deg}(x) \\
&= \sum_{v'\in F^{-1}(v)} \sum_{x\in f_{v'}^{-1}(e_v)} {\rm deg}(x)= \sum_{v'\in F^{-1}(v)}  {\rm deg}(v') = D_v.\end{align*}
 \end{proof}

Thus if $e$ is an edge connecting two vertices $v$ and $w$, then $D_v = D_e = D_w$. 
This number is constant, because the tree $T^Z$ is connected. It does not depend on $e$ neither on $v$. We denote by $D$ this number and call it the degree of ${\cal F}$. 

 \begin{figure}
 \centerline{\includegraphics[width=12cm]{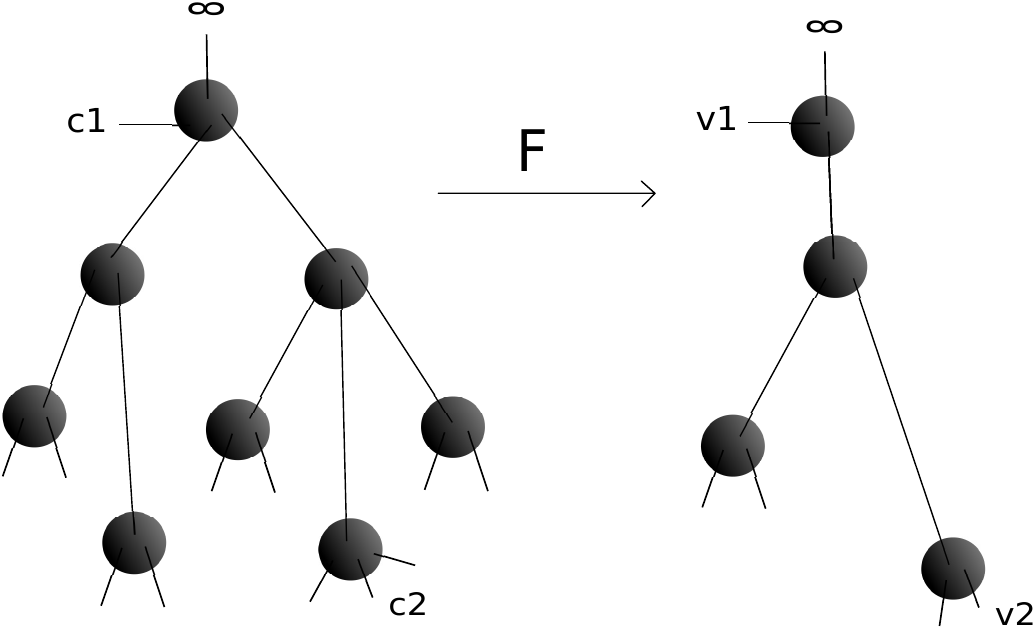}}
   \caption{A cover between trees of spheres of degree 3. The spheres on the left tree map horizontally to the spheres on the right tree. The sphere at the top on the left  maps by a degree three cover with a branching point of degree three at the attaching point of the branch containing $\infty$. The two spheres connecting $c2$ to the top sphere are mapped by covers of degree two to their images. The others are mapped by a degree one cover.}\label{excovertreepo}
 \end{figure}

\begin{corollary}
The map $F:T^Y\to T^Z$ is surjective.  
\end{corollary}

\begin{proof}
For every vertex $v$ of $T^Z$, we have $D_v\neq0$.
 \end{proof}

The following lemma and its corollary help to visualize the set of critical vertices distribution on a tree.

\begin{lemma}\label{ptscrit}    
Let ${\cal F}:{\cal T}^Y\to {\cal T}^Z$ be a cover between trees of spheres. Then every critical vertex lies in a path connecting two critical leaves, and each vertex on this path is critical.
\end{lemma}

\begin{proof} Let $v$ be a critical vertex of ${\cal F}$. Then $f_v$ has at least two distinct critical points. There are at least two distinct edges attached to $v$. So $v$ is on a path of critical vertices.

Let $[v_1,v_2]$ be such a path with a maximal number of vertices. From this maximality property, we see that there is only one critical edge (edge with degree strictly greater than one) attached to $v_1$. If $v_1$ is not a leaf then $f_{v_1}$ has just one critical point and that is not possible. So $v_1$ is a leaf. As well, $v_2$ is a leaf. 
 \end{proof}

Recall that the characteristic of a vertex $v$ of $T^X$ is $\chi_{T^X}(v):=2-\val(v)$, and thus it is equal to the Euler characteristic of ${\cal S}_v-X_v$.

We have a natural  Riemann-Hurwitz formula for covers between trees of spheres where $\chi_\T$ plays the same role as the Euler characteristic. 

\begin{proposition}[Riemann Hurwitz formula]\label{RH}
Let ${\cal F}:{\cal T}^Y\to {\cal T}^Z$ be a cover between trees of spheres of degree $D$, and $T''$ be a sub-graph of $T^Z$ and $T':= F^{-1}(T'')$. Then
\[\chi_{T^Y}(T') = D\cdot \chi_{T^Z}(T'') - \sum_{y\in \Crit{\cal F}\cap T'} \mult(y).\]
\end{proposition}

\begin{proof} If $v'' \in IV^Z$, then from the Riemann-Hurwitz formula we have
\begin{equation}\label{eqRH1}
\sum_{v'\in F^{-1}(v'')} \chi_{T^Y}(v') =  \sum_{v'\in F^{-1}(v'')}  \deg(v')\cdot \chi_{T^Z}(v) = D\cdot \chi_{T^Z}(v'').
\end{equation}
Otherwise, a leaf has characteristic $1$, so for every leaf $y$ of $T^Y$, we have 
\[\chi_{T^Y}(y) = \deg(y)-\mult(y).\]
Then, for every leaf $z\in Z$, we deduce that
\begin{equation}\label{eqRH2}
\sum_{y\in F^{-1}(z)} \chi_{T^Y}(z) = 
\sum_{y\in F^{-1}(z)}  \deg(y)-\mult(y) = D \cdot \chi_{T^Z}(z) - \sum_{y\in F^{-1}(z)}  \mult(y).
\end{equation}
By adding (\ref{eqRH1}) and (\ref{eqRH2}) for all vertices $v''\in IV^Z\cap T''$ and leaves $z\in T''\cap Z$, we get the formula. 
 \end{proof}

\paragraph{Consequences of the Riemann-Hurwitz formula}

\begin{corollary}\label{ptcrit}
 If ${\cal F} : {\cal T}^Y\to {\cal T}^Z$ is a cover between trees of spheres of degree D, then the tree $T^Y$ has $2D-2$ critical leaves counted with multiplicities.
\end{corollary}

\begin{proof}We use the Riemann-Hurwitz Formula for $T'=T^Y$ and $T''=T^Z$, and the fact that $\chi_{T^Y}(T^Y) = \chi_{T^Z}(T^Z)=2$ (cf Lemma  \ref{val}.)
 \end{proof}


We have proven that the pair $(F|_Y,\deg_{F}|_Y)$ defines a portrait (Definition \ref{Defportrait}).

\begin{corollary}\label{bornport}
If ${\cal F}:{\cal T}^Y\to {\cal T}^Z$ is a cover between trees of spheres of degree $D$, then
\[2-\card(Y) = D \cdot (2-\card(Z)).\]
\end{corollary}

\begin{proof} 
We apply the Riemann-Hurwitz formula for $T' = T^Y-Y$ and $T'' = T^Z-Z$, using the fact that $\chi_{T^Y}(T') = 2-\card Y$ and $\chi_{T^Z}(T'') = 2-\card(Z)$ (cf Lemma \ref{carautre}). Given that $T'$ doesn't have any leaf of $T^Y$, this proves the result.
 \end{proof}

We proved that the degree of ${\cal F}$ is bounded relatively to $\card(Y)$ and $\card(Z)$.

\begin{lemma} \label{restric}
Let ${\cal F}:{\cal T}^Y\to {\cal T}^Z$ be a cover between trees of spheres, $T''$ be an open, non empty and connected subset of $T^Z$ and $T'$ be a connected component of $F^{-1}(T'')$. Then the map $\overline {\cal F}:\overline {\cal T}'\to \overline {\cal T}''$ defined by 
\begin{enumerate}
\item $\overline F:= F: \overline T'\to \overline T''$ and 
\item $\overline f_v:= f_v$ if $v\in V'-Y'$
\end{enumerate}
is a cover between trees of spheres. 
\end{lemma}

\begin{proof}
Indeed, for every vertex $v\in V'-Y'$, edges on $v$ in $\overline T'$ are the same as the one on $T$ so $\overline f_v$ satisfies the required conditions. Moreover, leaves of $\overline T'$ are either leaves of $T$ and map to leaves of $T^Z$, so leaves of $\overline T''$, or are elements of $\overline T'-T'$. In this case they map to elements of $\overline T''-T''$ which are leaves of $\overline T''$ because adjacent vertices are mapped to adjacent vertices.
 \end{proof}

We define $\deg(\F|_{T'_Y}):=\deg\overline \F$ and $\mult T'_Y:=\deg(\F|_{T'_Y})-1$.

Then we have the restriction of the Riemann-Hurwitz formula to a connected component of the preimage.

\begin{proposition}\label{RH} Let ${\cal F}:{\cal T}^Y\to {\cal T}^Z$ be a cover between trees of spheres, $T''$ be a sub-tree of $T^Z$ and $T'$ be a connected component of $F^{-1}(T'')$. Then we have \[\chi_{T^Y}(T')=\deg(\F|_{\overline{\cal T'}})\cdot \chi_{T^Z}(T'')-\sum_{y\in \Crit\F\cap T'}\mult(y).\]
\end{proposition}

\begin{proof} 
Given that $\chi_{T^Y}(T') = \chi_{\widehat T'}(T')$ and $\chi_{T^Z}(T'') = \chi_{\widehat T''}(T'')$, the result follows immediately by using the Riemann-Hurwitz formula on the cover $\overline {\cal F} : \overline {\cal T}'\to \overline {\cal T}''$.
 \end{proof}


\section{Dynamics on stable trees}\label{chap3}

 In this section we suppose that $X\subseteq Y\cap Z$.

\subsection{Stable tree and dynamical system}

\begin{definition}[Stable tree]
A tree $T$ is stable if every internal vertex has valence greater than two. 
 \end{definition}
 
From now on, we suppose that all {\bf trees are stable}.

\begin{definition} 
 In a tree $T$, we say that a vertex $v$ separates three vertices $v_1$, $v_2$ and $v_3$ if the $v_i$ are in distinct connected components of $T-\{v\}$. 
  \end{definition}

 Note that three distinct vertices of $T$ lie either on a same path or they are separated by a unique vertex.

\begin{definition}[Compatible tree]\label{compdef}
A tree $T^X$ is compatible with a tree $T^Y$ if 
\begin{enumerate}
\item   $X\subseteq Y$, $IV^X\subseteq IV^Y$ and
\item for all vertices $v$, $v_1$, $v_2$ and $v_3$ of $V^X$, the vertex $v$ separates $v_1$, $v_2$ and $v_3$ in $T^X$ if and only if it does the same in $T^Y$.
\end{enumerate}
\end{definition}

Later in this article, it will be useful to know if a vertex is in $T^X$. The two following lemmas give a way to do this in some particular cases.

\begin{lemma} \label{definiX}
If $T^X$ is compatible with $T^Y$ and if an internal vertex $v\in IV^Y$ separates three vertices of $V^X$, then $v\in T^X$. 
\end{lemma}

\begin{proof}
Let $v_1$, $v_2$ and $v_3$ be these three vertices.
There is an internal vertex $v^X$ of $T^X$ separating $v_1$, $v_2$ and $v_3$ in $T^X$. From the compatibility we conclude that this vertex separates $v_1$, $v_2$ and $v_3$ in $T^Y$. It follows that $v^X=v$. 
 \end{proof}

Now we focus on \tss.

\begin{definition}
A \ts ${\cal T}^X$ is compatible with a \ts ${\cal T}^Y$ if 
\begin{enumerate}
\item $T^X$ is compatible with $T^Y$, 
\item for all internal vertex $v$ of $T^X$, we have 
\begin{enumerate}
\item $\St^X_v=\St^Y_v$ and
\item$a_v^X=a_v^Y|_X $. 
\end{enumerate}
\end{enumerate}
\end{definition}

If it is the case we write $\T^X\lhd \T^Y$. Now we can define a dynamical system of \tss. Note that when the spheres are equipped with a projective structure, then we will require in addition that  $\St^X_v$ and $\St^Y_v$ have the same one.

\begin{definition}[Dynamical systems]
A dynamical system of trees of spheres is a pair $(\F,\T^X)$ such that
\begin{enumerate}
\item ${\cal F}:{\cal T}^Y\to {\cal T}^Z$ is a cover between trees of spheres ;
\item $\T^X\lhd \T^Y$ and $\T^X\lhd\T^Z$. 
\end{enumerate}
\end{definition}

In \cite{A1} we prove that if such a $\T^X$ exists then it is unique.
Figure \ref{demarco} gives such an example of dynamical system.

%
%


\subsection{Dynamics on combinatorial trees}   \label{notations}
   
 \begin{figure}
  \centerline{\includegraphics[width=11.8cm]{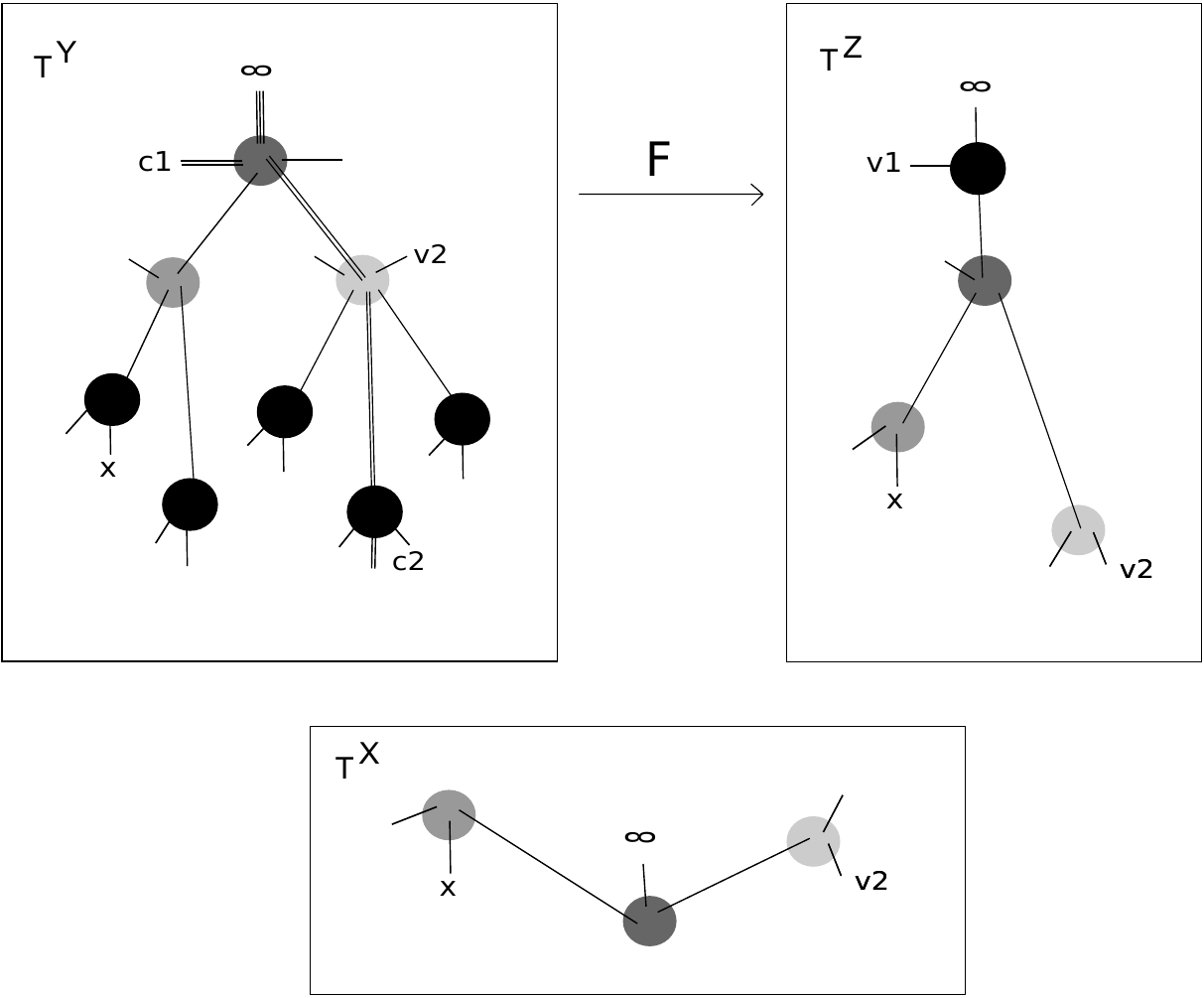}}
   \caption{The map $F$ is the one described on Figure \ref{excovertreepo}. On this example, internal vertices which are not in $\T^X$ are in black whereas internal vertices which are the same in $T^X$, $T^Y$ and $T^Z$ have the same color. The pair $(\F,\T^X)$ is a dynamical system. The internal vertex adjacent to $c_2$ maps to the closest vertex by a degree two cover. Then it maps to the dark grey one by the same kind of cover. Then it maps to the higher black one on $T^Z$ by a degree three cover. All other vertices map with degree one. On this example, each vertex cannot be iterated more than three times. }
\label{demarco} \end{figure}

As we have a common set $V^X$ in the trees $T^Y$ and $T^Z$, 
we can try now to iterate $F$ as soon as images stay in $V^X$. Recursively we define for $k\geq 1$
\[IV(F):=IV^X\quad\text{and}\quad 
IV(F^{k+1}):=\{v\in IV(F^k)~|~F^k(v)\in IV^X\}.\]
Let
\[\Prep(F):=\bigcap_{k\geq 1} IV(F^k).\]
If $v\in \Prep(F)$, then $F^k(v)$ is well defined and lies in $IV^X$ for all $k\geq 0$.

The set $\Prep(F)$ is finite and invariant under the map $F$, each vertex $v$ of $\Prep(F)$ is (pre)periodic under $F$. It may happen that $\Prep(F)$ is empty as we can see on the example on Figure \ref{demarco}.

 If $v\in IV^Y{-} \Prep(F)$, then there exists a smallest integer  $k\in \N$ such that $F^k(v)\notin V^X$. We say that $v$ is forgotten by $F^k$ or simply that $v$ is forgotten if $k=0$.
On Figure \ref{demarco}, internal vertices of $T^Y$ at the bottom on $\T^Y$ are forgotten by $F^3$.

Restricting the dynamic on vertices would be ignoring the tree structure. The following lemma shows a strong restriction coming from the compatibility.

\begin{lemma}\label{identifiable}
Let $B\subset T^Z$ be a branch on $v\in V^X$. If $B$ contains a vertex of $V^X$ then its attaching point $i^Z_v(B)$ lies in $X_v$. 
\end{lemma}

\begin{proof}
Either this vertex is a leaf and the result is trivial, or it is not a leaf and $B$ contains a leaf of $T^X$, then we are in the previous case.
 \end{proof}
 

\subsection{Dynamics on \tss}

If ${\cal F}:{\cal T}^Y\to {\cal T}^Z$ is a cover between trees of spheres and if in addition $v\in T^Y$ then we will denote by $f_v^k$ the composition $f_{F^{k-1}(v)}\circ\ldots\circ f_{F(v)}\circ f_v$ as soon as $v\in IV(F^{k'})$ for some $k'>k$.
We define $\Sigma$ to be the disjoint union of the $\St_v$ for $v\in T^Y$.
The orbit of any point $z$ in $\Sigma$ is the set 
\[\cal{O}(z) :=\{f_v^k(z)\;|\;k\geq0,z\in\St_v,v\in IV(F^{k})\}.\] 
Two points of $\Sigma$ are in the same grand orbit if their orbits intersect. 
Let $\mathcal {GOC}^{\infty}$ denote the set of grand orbits containing a critical point with infinite forward orbit.

 \begin{theorem}[Spheres periodic cycles]\label{npcf}
Let $(\F,\T^X)$ be a dynamical system of \tss. Then, 
$\card\mathcal {GOC}^{\infty} \leq 2 \deg(\F)-2$. 
 \end{theorem}

\begin{proof}
Let $c\in \Sigma$ be a critical point of the map $f$ with grand orbit in $\mathcal {GOC}^{\infty}$. Then $c$ lies in a sphere ${\cal S}_v$ with $v\in \Prep(F)$.
For $k\geq 0$, we define $v_k:= F^{ k}(v)$ and $c_k:=f^k(c)$. 
We have $\card \{c_k\}=\infty$, so there exists $k_0\geq 1$ such that 
\begin{enumerate}
\item $c_k$ is an attaching point of an edge in ${\cal T}^Y$ for $k <k_0$ and 
\item $c_{k_0}$ is not the attaching point of an edge of $T^Y$
\end{enumerate} (indeed, the number of edges attaching points in ${\cal T}^Y$ is finite).

For $k\in [0,{k_0}-1]$, we define 
\begin{enumerate}
\item $B_k^Y$ to be the branch of $T^Y$ on $v_k$ attached to $c_k$, 
\item $B_{k+1}^Z$ to be the branch of $T^Z$ on $v_{k+1}$ attached to $c_{k+1}$ and 
\item $\widetilde B_k = B_k^Y\cap F^{-1}(B_{k+1}^Z)$. 
\end{enumerate}
Let $k_1\geq 1$ be the minimal integer such that $\widetilde B_k = B_k^Y$ for $k\in [k_1,{k_0}-1]$. We define
\[B_c:=\bigcup_{k_1-1}^{{k_0}-1} \widetilde B_k.\]
Given that $c_{k_0}$ is not an attaching point of $T^Y$, every vertex of $B_c$ is forgotten by an iterate of $F$. In other words, $B_c\cap \Prep(F)=\emptyset$. 

\begin{lemmaint}
The open set $B_c$ contains a critical leaf. 
\end{lemmaint}

\begin{proof}
Either $k_1=0$ and $B_0^Y= \widetilde B_0$. 
From the Riemann-Hurwitz formula, we have 
\begin{align*}1=\chi_{T^Y}(\widetilde B_0) &= \deg(\F:\widetilde B_0\to B_1^Z) \cdot \chi_{T^Z}(B_1^Z) - \mult(\widetilde B_0) \\&\geq \deg(\F:\widetilde B_0\to B_1^Z)- \mult(B_c).\end{align*}

Given that $c$ is a critical point of $f_v$, we have \[\deg(F:\widetilde B_0\to B_1^Z)\geq \deg(v)\geq 2.\] 
So $\mult(B_c)\geq 1$ and $B_c$ contains at least a critical leaf. 

We know that $k_1\geq 1$ and that $\widetilde B_{k_1-1}$ is not a branch. According to Lemma \ref{branch}, we decuct that ${\chi_{T^Y}(\widetilde B_{k_1-1})\leq 0}$. 
From the Riemann-Hurwitz formula, we have
\begin{align*}
0\geq \chi_{T^Y}(\widetilde B_{k_1-1})&= \deg(\F:\widetilde B_{k_1-1}
\to B_{k_1}^Z)\cdot \chi_{T^Z}(B_{k_1}^Z)  - \mult(\widetilde B_{k_1-1})\\& \geq 1- \mult(B_c).
\end{align*}

So $\mult(B_c)\geq 1$ and $B_c$ contains at least a critical leaf.  \end{proof}

\begin{lemmaint}
Let $c\in \Sigma$ and $c'\in \Sigma$ be two attaching points with infinite disjoint orbits. Then $B_c\cap B_{c'}=\emptyset$. 
\end{lemmaint}

\begin{proof}
If $B_c\cap B_{c'}\neq \emptyset$, then $F(B_c)\cap F(B_{c'})\neq \emptyset$ and we can find two integers $k$ and $k'$ such that the branch of $T^Z$ attached to $c_{k}$ intersect the branch of $T^Z$ attached to $c'_{k'}$. In this case, 
\begin{enumerate}
\item Either $v_k=v'_{k'}$ and $c_k=c'_{k'}$, which contradicts the fact that orbits of $c$ and $c'$ are disjoint; 
\item either $v_{k}$ lies in the branch of $T^Z$ attached to $v'_{k'}$. As $\Prep(F)\cap V^Z\subset V^X$, we have $B_{c'}\cap \Prep(F)=\emptyset$ which contradicts Lemma \ref{identifiable}; 
\item or $v'_{k'}$ lies in the branch of $T^Z$ attached to $v_{k}$. As $\Prep(F)\cap V^Z\subset V^X$, we have $B_{c}\cap \Prep(F)=\emptyset$ which contradicts Lemma \ref{identifiable}. 
\end{enumerate}
 \end{proof}

Now we finish proof of Theorem \ref{npcf}.
Let $c_1$, \ldots, $c_N$ be critical points with disjoint and infinite orbits. The open sets $B_{c_1}$, \ldots, $B_{c_N}$ are disjoint and each one contains a critical leaf of $T^Y$. The number of critical leaves counted with multiplicity is $2\deg({\cal F})-2$ so $N\leq 2\deg({\cal F})-2$. 
 \end{proof}

\begin{definition}\label{deftriviall} 
Let $(\F,\T^X)$ be a dynamical system of \tss.
A cycles of spheres is of $(\F,\T^X)$ is not post-critically finite if one of the spheres of this cycle contains a critical point with an infinite orbit.
\end{definition}

\begin{proof}(Theorem \ref{thmkiw1}) 
There is a critical point with grand orbit in $\mathcal {GOC}^{\infty}$ lying on each cycles of spheres which are not post-critically finite so by Theorem \ref{npcf} there are at most $2D-2$ cycles of spheres which are not post-critically finite.
 \end{proof}
 


\section{Convergence notions}\label{chap4}

\label{chapconv}

Recall that we are supposing that trees are {\bf stable}. Here we also require that the trees are {\bf projective} and that all the covers are {\bf holomorphic} in a sense that we define below. We denote by ${\S:=P^1(\C)}$ the Riemann sphere.

In this chapter, we define a notion of convergence on the set of trees of spheres. This notion is not Hausdorff, but in \cite{A1} we show that it corresponds to a Hausdorff topology on the natural quotient of this set under the action of isomorphisms of trees of spheres. 


\subsection{Holomorphic covers}

\begin{definition}[Projective structure]
A projective structure on a \ts $\T$ marked by $X$ is the data for every $v\in IV$ of a projective structure on $\St_v$. 
\end{definition}

According to the Uniformization Theorem, giving a complex structure on $\St_v$ is the same as giving a class of homeomorphisms $\sigma: \St_v \to  \S $ where $\sigma$ is equivalent to $\sigma'$ when $\sigma'\circ\sigma^{-1}$ is a Moebius transformation. Such a $\sigma$ is called a projective chart on $\St_v$.
When the topological sphere $\cal S_v$ has such a projective structure, we will denote it by $\mathbb S_v$.

\begin{definition}[Holomorphic covers]
A cover between trees of spheres $\F:\T^Y\to\T^Z$ with a given projective structure is holomorphic if for all internal vertex $v$, the map $f_v:\S_v\to\S_{F(v)}$ is holomorphic.
\end{definition}

If $f_v$ is holomorphic then its expression in projective charts is a rational map.

When a tree of spheres is compatible to another one, we require that the projective structures on a common sphere are the same.


\subsection{Marking rational maps}\label{herearedef}

 For $d\geq 1$, we denote by $\Rat_d$ the set of rational maps of degree $d$. In particular, $\Aut(\S):=\Rat_1$ is the set of Moebius transformations. This set acts on $\Rat_d$ by conjugacy :
\[\Aut(\S)\times \Rat_d\ni (\phi,f)\mapsto \phi\circ f\circ \phi^{-1}\in \Rat_d.\]  
We denote by $\rat_d$ the quotient of $\Rat_d$ by this action.

\begin{definition}[Marked sphere]
A sphere marked (by $X$) is an injection \[x:X\to \mathbb S.\]
\end{definition}

Each tree of spheres having a single internal vertex $v$ is identified with the marked sphere $i_v$.

\begin{recall}
A portrait ${\bf F}$ of degree $d\geq 2$ is a pair $(F,\deg)$ where 
\begin{enumerate}
\item $F:Y\to Z$ is a map between finite sets $Y$ and $Z$ and
\item $\deg:Y\to \N-\{0\}$ is a function verifying
\[\sum_{a\in Y}\bigl(\deg(a) -1\bigr) = 2d-2\quad\text{and}\quad \sum_{a\in F^{-1}(b)} \deg(a) = d\quad\text{ for all } b\in Z.\] 
\end{enumerate}
\end{recall}

\begin{definition}[Marked rational maps]\label{exfondsph00}
A rational map marked by the portrait ${\bf F}$ is a triple $(f,y,z)$ where
\begin{enumerate}
\item $f\in \Rat_d$
\item $y:Y\to \S$ and $z:Z\to \S$ are marked spheres, 
\item $f\circ y = z\circ F$ on $Y$ and 
\item $\deg_{y(a)}f = \deg(a)$ for $a\in Y$, 
\end{enumerate}
\end{definition}

%
%

If $(f,y,z)$ is marked by ${\bf F}$, we have the following commutative diagram : 

\centerline{
$\xymatrix{
 Y \ar[r]^{y}    \ar[d] _{{ F}} &\S  \ar[d]^{f} \\
      Z \ar[r]_{z}  &\S.
  }$}

Typically, $Z\subset \S$ is a finite set, $F:Y\to Z$ is the restriction of a rational map $F:\S\to \S$ to $Y:=F^{-1}(Z)$, and $\deg(a)$ is the local degree of $F$ at $a$. In this case, the Riemann-Hurwitz formula and the conditions on the function $\deg$ imply that $Z$  contains the set $V_F$ of the critical values of $F$ so that $F:\S-Y\to \S-Z$ is a cover. 


A cover between trees of spheres ${{\cal F}:{\cal T}^Y\to {\cal T}^Z}$ such that $\T^Y$ and $\T^Z$ are marked spheres (with respective unique internal vertices  $v$ and $v'$) is the same data as a branched cover between $\S_v$ and $\S_{v'}$ such that the set of attaching points on $\S_v$ is the pre-image of the set of attaching points on $\S_{v'}$ and contains the branching locus.
Such a cover is called a marked spheres cover. We identify $\F$ with the marked rational map $(f_v,a_v^Y,a_{v'}^Z)$.

\begin{definition}[Dynamically marked rational map]\label{defdefdef06}
A rational map dynamically marked by  $({\bf F},X)$ is a rational map $(f,y,z)$ marked by  ${\bf F}$ such that  $y|_X=z|_X$. 
\end{definition}

We denote by 
 $\Rat_{{\bf F},X}$ the set of rational maps dynamically marked by $({\bf F},X)$.

Let $(\F:\T^Y\to\T^Z,\T^X)$ be a dynamical system such that $\F$ is a cover of marked spheres. The tree $T^X$ has a unique internal vertex and given that $T^Y$ and $T^Z$ have only one internal vertex, then the one of $T^X$ is the same as $v$, the one of $T^Y$, and of $T^Z$. Then we identify $(\F,\T^X)$ and the dynamically marked rational map $(f_v,a_v^Y,a_{v}^Z)$. We say that it is a dynamical system of spheres marked by ${\bf F}:= (F,\deg)$.



\subsection{Convergence of marked spheres}


\begin{definition}
A sequence of marked spheres $x_n:X\to{\mathbb S}_n$
  converges to a \ts ${\cal T}^X$
  if for all internal vertex $v$
 of $\T^X$ , there exists a (projective) isomorphism $\phi_{n,v}:{\mathbb S}_n\to{\S}_v$
    such that $\phi_{n,v} \circ x_n$
     converges to $a_v$. 
     \end{definition}
(We prefer to use the notation $\S_n$ instead of $\S$ because the $\S_n$ can be distinct.)
We will use the notation ${x}_n\to {\cal T}^X$ or $\displaystyle {x}_n\underset{\phi_n}\longrightarrow {\cal T}^X$. 

In this paper we assume the following result that appears in \cite{FT,A,A1}:

\begin{theorem}\cite[Corollaire 4.22]{A}\label{thmcomp}
Given a finite set $X$ with at least three elements, every sequence of spheres marked by $X$ converges, after extracting a subsequence, to a tree of spheres marked by $X$.
\end{theorem}

\begin{example}Suppose that $X:=\{\chi_1,\chi_2,\chi_3,\chi_4\}$. For all $n\geq 1$, let ${x}_n:X\to \S$ be the marked sphere defined by :
\[x_n(\chi_1):=0, \quad x_n(\chi_2):=1,\quad x_n(\chi_3):=n\quad\text{and}\quad x_n(\chi_4):=\infty.\]
Let ${ \T}^X$ be the tree of projective spheres marked by $X$ with two distinct internal vertices $v$ and $v'$ of valence $3$ with $\S_v:=\S_{v'}:=\S,$
\[a_v(\chi_1):=0,\quad a_v(\chi_2):=1,\quad a_v(\chi_3):=a_v(\chi_4):=\infty,\]
\[a_{v'}(\chi_1):=a_{v'}(\chi_2):=0,\quad a_{v'}(\chi_3):=1\quad\text{and}\quad a_{v'}(\chi_4):=\infty.\]

Considering the isomorphisms $\phi_{n,v}:\S\to \S_v$ and $\phi_{n,v'}:\S\to \S_{v'}$ defined by : 
\[\phi_{n,v}(z):=z\quad\text{and}\quad \phi_{n,v'}(z):=z/n \text{ (cf Figure \ref{remfond}), }\]
we prove that $\displaystyle  {x}_n\underset{\phi_n}\longrightarrow { \T}^X$. 
\end{example}

 \begin{figure}
  \centerline{\includegraphics[width=7cm]{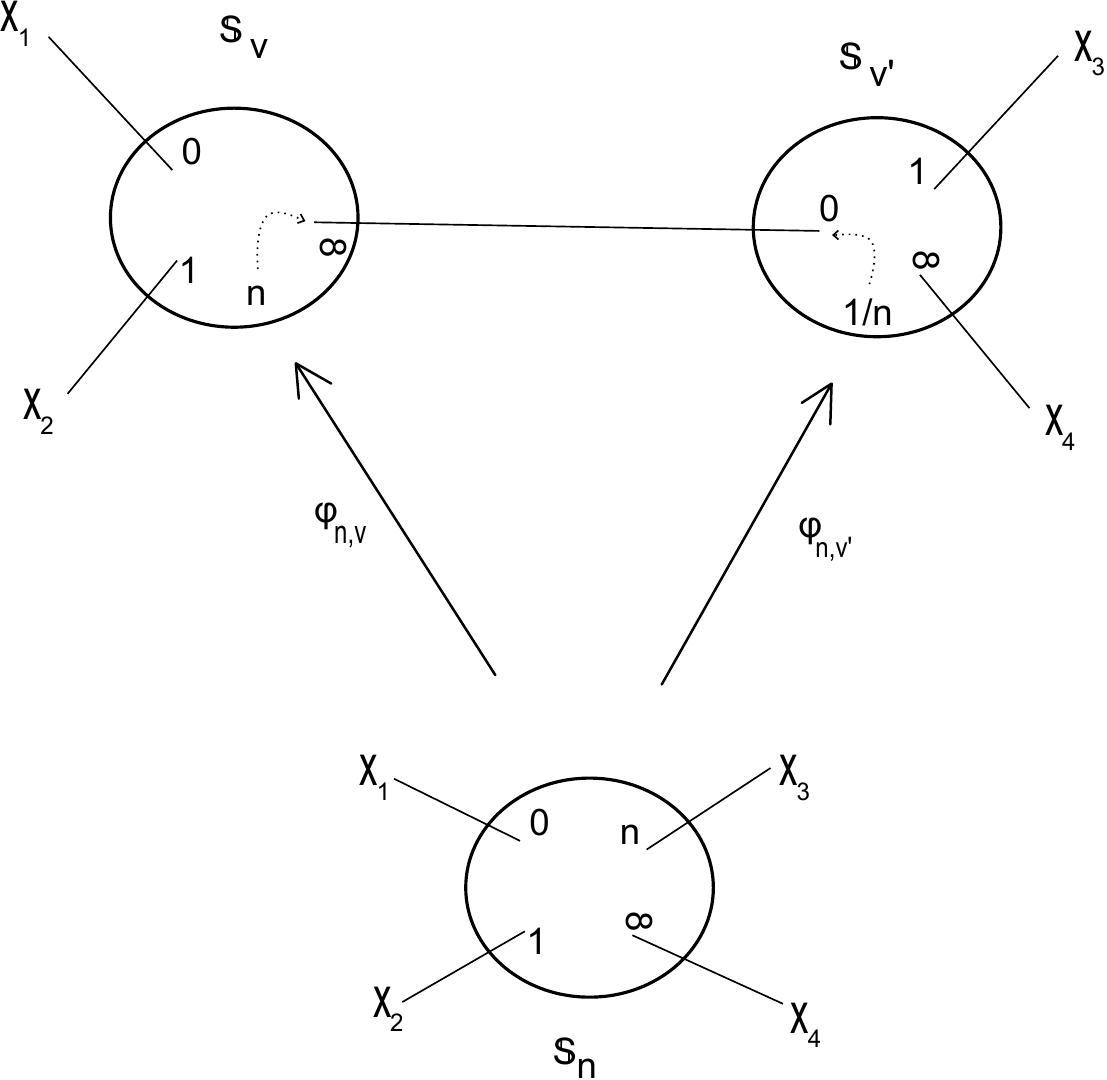}}
   \caption{ }
\label{remfond} \end{figure}

\begin{lemma} \label{noncomp}
Let $v$ and $v'$  be two distinct internal vertices of $\T^X$,  
and consider a sequence of marked spheres $(x_n)_n$ such that $x_n\underset{\phi_n}\longrightarrow \T^X$. 
Then the sequence of isomorphisms $(\phi_{n,v'} \circ \phi_{n,v}^{-1})_n$
converges locally uniformly outside $i_v(v')$ to the constant $ i_{v'}(v)$. 
\end{lemma}

\begin{proof}
Each vertex $v$ and $v'$ has three edges and every branch has at least a leaf, so there exist four marked points $\chi_1,\chi_2,\chi_3,\chi_4\in X$ 
such that $v$ separates $\chi_1$, $\chi_2$ and $v'$, and the vertex $v'$ separates $\chi_3$, $\chi_4$ and $v$.

We define for $j\in \{1,2,3,4\}$, 
\[\xi_j:=a_v(\chi_j),\quad \xi'_j:=a_{v'}(\chi_j),\quad 
\xi_{j,n}:=\phi_{n,v}\circ x_n(\chi_j)\quad \text{and}\quad 
\xi'_{j,n}:=\phi_{n,v'}\circ x_n(\chi_j).\]
From the hypothesis, $\xi_{j,n}\to \xi_j$ and $\xi'_{j,n}\to \xi'_j$ when $n\to \infty$. Moreover, $\xi_3=\xi_4=i_v(v')$ and $\xi'_1=\xi'_2=i_{v'}(v)$. 
Even if we must post-compose $\phi_{n,v}$ and $\phi_{n,v'}$ by automorphisms of $\S_v$ and $\S_{v'}$ that are converging to the identity when $n\to \infty$ and don't change the limit of $\phi_{n,v'}\circ \phi_{n,v}^{-1}$, we can suppose that for all $n$,
\[\xi_{1,n}=\xi_1,~\xi_{2,n}=\xi_2,~\xi_{3,n}=\xi_3,~\xi'_{1,n}=\xi'_1,~\xi'_{3,n}=\xi'_3\text{ and }\xi'_{4,n}=\xi'_4.\]

Now we consider  the projective charts $\sigma$ on $\S_v$ and $\sigma'$ on $\S_{v'}$ defined by :
\begin{enumerate}
\item $\sigma(\xi_1)=0$, $\sigma(\xi_2)=1$ and $\sigma(\xi_3)=\infty$;
\item $\sigma'(\xi'_1)=0$, $\sigma'(\xi'_4)=1$ and $\sigma'(\xi'_3)=\infty$.
\end{enumerate}
The Moebius transformation $M_n:=\sigma'\circ \phi_{n,v'}\circ \phi_{n,v}^{-1}\circ \sigma^{-1}$ fixes $0$ and $\infty$ and maps $\sigma(\xi_4)$ to $1$. Thus 
\[M_n(z) = \frac{z}{\lambda_n}\quad\text{with}\quad
\sigma(\xi_{4,n}) \underset{n\to \infty}\longrightarrow \infty.\]
Consequently, $M_n$ converges locally uniformly outside infinity to the constant map equal to zero. Then, $\phi_{n,v'}\circ \phi_{n,v}^{-1}={\sigma'}^{-1}\circ M_n\circ\sigma$ converges locally uniformly to the constant $({\sigma'})^{-1}(0)=i_{v'}(v)$ outside $\sigma^{-1}(\infty)=i_v(v')$. 
 \end{proof}


\subsection{Convergence of marked spheres covers}

To each marked rational map $(f,y,z)$, we can associate a cover between trees of spheres from a sphere marked by $Y$, via the map $y$, to a sphere marked by $Z$ via the map $z$.


\begin{definition}[Non dynamical convergence]
Let ${ \F}:{\T}^Y\to { \T}^Z$ be a cover between trees of spheres of portrait ${\bf F}$. A sequence $(f_n,y_n,z_n)_n$ of marked spheres covers converges to ${ \F}$ if their portrait is ${\bf F}$ and if for all pair of internal vertices $v$ and $w:=F(v)$, there exists sequences of isomorphisms $\phi_{n,v}^Y:\S_n^Y\to \S_v$ and $\phi_{n,w}^Z:\S_n^Z\to \S_w$ such that 
\begin{enumerate}
\item $\phi_{n,v}^Y\circ y_n:Y\to \S_v$ converges to $a_v^Y:Y\to \S_v$, 
\item $\phi_{n,w}^Z\circ z_n:Z\to \S_w$ converges to $a_w^Z:Z\to \S_w$ and 
\item $\phi_{n,w}^Z\circ f_n\circ (\phi_{n,v}^Y )^{-1}:\S_v\to \S_w$ converges locally uniformly outside $Y_v$ to ${f_v:\S_v\to \S_w}$. 
\end{enumerate}
\end{definition}

We use the notation $(f_n,y_n,z_n)\rightarrow \F$ or $(f_n,y_n,z_n)\underset{(\phi^Y_n,\phi^Z_n)}\longrightarrow  \F.$

In this paper we assume the following result appearing in \cite{A1}, as a reformulation of \cite[Th\'eor\`eme 5.9]{A}.

\begin{theorem}\label{thmcomp00}
Let $y_n,z_n$ be two sequences of spheres marked respectively by the finite sets $Y$ and $Z$ each one containing at least three elements and converging to the trees of spheres $\T^Y$ and $\T^Z$. 

Then, every sequence of marked spheres covers $(f_n,y_n,z_n)_n$ of a given portrait converges to a cover between the trees of spheres $\T^Y$ and $\T^Z$.
\end{theorem}

Note that we have the following lemma.

\begin{lemma}\label{cvu}
Let ${\cal F}:{\cal T}^Y\to{\cal T}^Z$ be a cover between trees of spheres with portrait ${\bf F}$ and of degree $D$. Let
$v\in IV^Y$ with $\deg(v)=D$ and ${(f_n,y_n,z_n)_n}$ be a sequence of marked spheres covers that satisfies $\displaystyle (f_n,y_n,z_n)\underset{\phi^Y_n,\phi^Z_n}\longrightarrow{\cal F}$. Then the sequence $\phi_{n,F(v)}^Z\circ f_n\circ (\phi_{n,v}^Y)^{-1}:\S_v\to \S_{F(v)}$ converges uniformly
to $f_v:\S_v\to \S_{F(v)}$.
\end{lemma}

\begin{proof}
We define $w:=F(v)$. We choose the projective charts ${\sigma_v:\S_v\to \S}$ and $\sigma_w:\S_w\to \S$ such that no point of $Y_v$ or of
$Z_w$ maps to infinity. We define
\[g_n:=\sigma_w\circ \phi_{n,w}\circ f_n\circ \phi_{n,v}^{-1}\circ
\sigma_v^{-1}\quad
\text{and}\quad g:=\sigma_w\circ f_v\circ \sigma_v^{-1}.\]
We supposed that the sequence $(g_n)_n$ converges locally uniformly to $g$ out
of $\sigma_v(Y_v)$. All the $D$ poles of $g_n$ (counting
multiplicities) converge to the $D$ poles of $g$. In particular, if $U$ is a sufficiently small neighborhood of
$\sigma_v(Y_v)$, then
\begin{enumerate}
\item for $n$ large enough, $g_n$ is holomorphe without poles in $U$ and
\item $g_n-g$ converges uniformly to $0$ on the boundary of $U$.
\end{enumerate}
By the maximum modulus principle, $(g_n-g)_n$ converges uniformly
to $0$ in $U$. So $(g_n)_n$ converges locally
uniformly to $g$ in the neighborhood of points of $\S$ and given that $\S$
is compact, then $g_n$ converges uniformly to $g$ on $\S$.
 \end{proof}

\subsection{Dynamical convergence of marked spheres covers}

\begin{definition}[Dynamical convergence]\label{defcvdyn}
Let $({\F}:{ \T}^Y\to { \T}^Z,{ \T}^X)$ be a dynamical system of trees of spheres with portrait ${\bf F}$. A sequence $(f_n,y_n,z_n)_n$ of dynamical systems between spheres marked by $({\bf F},X)$
converges to $(\F,\T^X)$ if  
\[\displaystyle (f_n,y_n,z_n)\underset{\phi_n^Y,\phi_n^Z}\longrightarrow{ \F}\quad\text{with}\quad\phi_{n,v}^Y=\phi_{n,v}^Z\] for all vertex $v\in IV^X$. We say that $(\F,\T^X)$ is dynamically approximable by $(f_n,y_n,z_n)_n$.
\end{definition}
We use the notation $\displaystyle (f_n,y_n,z_n)\overset{\lhd}{\underset{\phi_n^Y,\phi_n^Z}\longrightarrow}{ \F}$ or simply $\displaystyle (f_n,y_n,z_n)\overset{\lhd}{\longrightarrow}{ \F}.$

We denote by $\partial\Rat_{{\bf F},X}$ the set of dynamical systems of trees of spheres which are approximable by a sequence in $\Rat_{{\bf F},X}$ and which are not in $\Rat_{{\bf F},X}$.
We use the notation $\phi_n$ instead of $\phi^\star_n$ when there is not possible confusion.

Note that requiring a dynamical convergence is not something very strong because we can prove the following:

\begin{lemma}\label{dynconvbis}
If $\displaystyle(f_n,y_n,z_n)\underset{\phi_n^Y,\phi_n^Z}\longrightarrow{ \F}$ and $(\F,\T^X)\in\Rat_{{\bf F},X}$, then there exists $(\tilde\phi_n^Z)_n$
such that $\displaystyle (f_n,y_n,z_n)\overset{\lhd}{\underset{\phi_n^Y,\tilde\phi_n^Z}\longrightarrow}{ \tilde\F}$ with $\tilde F=F.$
\end{lemma}

\begin{proof}
It is sufficient to change for every $v\in IV^X$ the map $\phi_{n,v}^Z$ for $\phi_{n,v}^X$ in the collection $\phi_n^Z$. 

Indeed, take $w\in IV^X$.
We have $a_w^X=a_w^Z|_X $ so as $X_w$ contains at least three elements, we deduce that $\phi_{n,w}^X\circ (\phi_{n,w}^Z)^{-1}$ converges uniformly to a Moebius transformation $M$. Then, as $w=F(v)$, the map $\phi_{n,w}^Z\circ f_n\circ (\phi_{n,v}^Y )^{-1}$ converges locally uniformly outside $Y_v$ to $f_v$, it is the same for $\phi_{n,w}^X\circ f_n\circ (\phi_{n,v}^Y )^{-1}$ which converges uniformly to $M\circ f_v$.
 \end{proof}

\begin{corollary}[of Lemma \ref{cvu}]\label{cvuutil}
Let $({\cal F},{\cal T}^X)$ be a dynamical system of trees of spheres of degree $d$, dynamically approximable by $(f_n,y_n,z_n)_n$. Suppose that $v\in IV^X$ is a fixed vertex such that $\deg(v) = d$. Then the sequence $[f_n]\in \rat_d$ converges to the conjugacy class $[f_v]\in \rat_d$. 
\end{corollary}

\begin{lemma}\label{convit}
Let $(\F,\T^X)\in\Rat_{{\bf F},X}$ be dynamically approximable by $(f_n,y_n,z_n)_n$. If $v\in IV(F^k)$ and if $w:=F^k(v)$, then $(\phi_{n,w}\circ f_n^k\circ \phi_{n,v}^{-1})_n$ converges locally uniformly to $f_v^k$ outside a finite number of points.
\end{lemma}

\begin{proof}
Indeed, it is sufficient to note that 
\[\phi_{n,v'}\circ f_n^k\circ \phi_{n,v}^{-1}=\phi_{n,v'}\circ f_n\circ \phi_{n,F^{k-1}(v)}^{-1}\ldots\circ\phi_{n,F^2(v)}\circ f_n\circ \phi_{n,F(v)}^{-1}\circ\phi_{n,F(v)}\circ f_n\circ \phi_{n,v}^{-1}\] 
so there is local uniform convergence as soon as the domain iterated does not intersect any attaching point of any edge.
 \end{proof}


\section{Rescaling-limits}\label{chap5}

\subsection{Definitions}\label{indepresclim}

In this section we recall the definition of rescaling limits introduced by Jan Kiwi in \cite{Kiwi2}.

%
%
%

\begin{definition}
For a sequence of rational maps $(f_n)_n$ of a given degree, a rescaling is a sequence of Moebius transformations $(M_n)_n$ such that there exist $k\in\N$ and a rational map $g$ of degree $\geq 2$ such that
\[M_n\circ f_n^k\circ M_n^{-1}\to g\]
uniformly on compact subsets of $\S$ with finitely many points removed.

If this $k$ is minimum, then it is called the rescaling period for $(f_n)_n$ at $(M_n)_n$ and $g$ is called a rescaling limit for $(f_n)_n$. Moreover $[g]\in\rat_{\deg\, g}$ is called a rescaling limit of the sequence $([f_n])_n$ in $\rat_d$.
\end{definition}

For example the family
$[f_\epsilon:z\mapsto\epsilon (z+1/z)]$ diverges in ${\rm rat}_d$ as $\epsilon$ tends to zero. For these representatives, we have $f_\epsilon\to 0$ as $\epsilon\to 0$ but the family of the second iterates has a non-constant limit:
\[f^2_\epsilon(z)=\epsilon\left(\epsilon (z+{1}/{z})+\frac{1}{\epsilon (z+1/z)}\right)\underset{\varepsilon\to 0}{\longrightarrow} \frac{z}{z^2+1}.\]

This phenomenon is possible because $\rat_d$ is not compact for $d\geq2$. More generally, consider a diverging sequence of conjugacy classes of rational maps in ${\rm rat}_d$. 
The limits of sub-sequences of representatives $(f_n)_n$ are constant maps or maps with degree strictly less than $d$. Sometimes we can have an integer $k\geq 1$ such that $(f_n^k)_n$ converges to a function $f$ which is not constant (thus dynamically interesting) even if every sub-sequence converges to a constant. 

In \cite{Kiwi2}, Jan Kiwi gives various examples of such behaviors and a  historical account on this topic. For his study he uses the formalism of Berkovich spaces in the spirit of \cite{Juan}, \cite{Kiwi3} and \cite{Kiwi1}.

Note that naturally we are interested in sequences in $ \rat_d$, so there is an equivalence relation associated to rescalings if we want to look at rescaling limits in their natural quotient space ($[g]\in \rat_{\deg \,g}$) which is the one defined below.

\begin{definition}[Independence and equivalence of rescalings]
Two rescalings $(M_n)_n$ and $(N_n)_n$ of a sequence of rational maps $(f_n)_n$ are independent if $N_n\circ M_n^{-1}\to \infty$ in $\Rat_1$. That is, for every compact set $K$ in $\Rat_1$, the sequence $N_n\circ M_n^{-1}\notin K$ for $n$ big enough.
They are said to be equivalent if $N_n\circ M_n^{-1}\to M$ in $\Rat_1$.
\end{definition}

\begin{definition}[Dynamical dependence]
Given a sequence $(f_n)_n\in\Rat_d$ and given $(M_n)_n$ and $(N_n)_n$ of period dividing $q$. We say that $(M_n)_n$ and $(N_n)_n$ are dynamically dependent if, for some subsequences $(M_{n_k})_{n_k}$ and $(N_{n_k})_{n_k}$, there exist $1\leq m\leq q$, finite subsets $S_1,S_2$ of $\S$ and non constant rational maps $g_1,g_2$ such that
 \[L^{-1}_{n_k}\circ f^m_{n_k}\circ M_{n_k}\to g_1\]
uniformly on compact subsets of $\S\setminus S_1$ and
\[M^{-1}_{n_k}\circ f^{q-m}_{n_k}\circ L_{n_k}\to g_2\]
uniformly on compact subsets of $\S\setminus S_2$.
\end{definition}


\subsection{From trees of spheres to rescaling-limits}

 Below we explain the relation between rescaling limits and dynamical systems between trees of spheres approximable by a sequence of dynamical systems of marked spheres.

\begin{theorem}\label{lienkiwi0}
Let ${\bf F}$ be a portrait, $(f_n,y_n,z_n)_n\in \Rat_{{\bf F},X}$ and $(\F,\T^X)$ be a dynamical system of trees of spheres. Suppose that
\[\displaystyle { f}_n\overset{\lhd}{\underset{\phi_n^Y,\phi_n^Z}\longrightarrow}{ \F}.\]

If $v$ is a periodic internal vertex in a critical cycle with exact period $k$, then ${f^k_v:\S_v\to\S_v}$ is a rescaling limit corresponding to the rescaling $(\phi^Y_{n,v})_n$. 

In addition, for every $v'$ in the cycle, $(\phi^Y_{n,v'})_n$ and $(\phi^Y_{n,v})_n$ are dynamically dependent rescalings.
\end{theorem}

%
%
%
%
%

\begin{proof}
If $v$ is a periodic internal vertex in a critical cycle with exact period $k$, then according to Lemma \ref{convit}, $(\phi^Y_{n,v}\circ f_n^k\circ (\phi^Y_{n,v})^{-1})_n$ converges locally uniformly to
${f^k_v:\S_v\to\S_v}$ so $(\phi^Y_{n,v})_n$ is a rescaling and the rescaling limit is ${f^k_v}$.

Again, according to  Lemma \ref{convit}, if $0<k'<k$, then $(\phi_{n,F{k'}(v)}\circ f_n^{k'}\circ \phi_{n,v}^{-1})_n$ and $(\phi_{n,v}\circ f_n^{k-k'}\circ \phi_{n,F^{k'}(v)}^{-1})_n$ converge respectively and locally uniformly outside finite sets to 
$f^{k'}_{v}$ and $f^{k-k'}_{F^{k}(v)}$, so the rescalings $( \phi_{n,v})_n$ and $(\phi_{n,F{k'}(v)})_n$ are dynamically dependent.
 \end{proof}


\subsection{From rescaling-limits to trees of spheres}

In this section, we explore the reciprocal question: if there exist rescaling limits, does there exists a dynamical systems between trees of spheres such that these rescalings correspond to spheres in critical periodic cycles as described in the previous section? 
The following theorem gives the answer.

\begin{theorem}\label{intro3}
Given a sequence $(f_n)_n$ in $\Rat_d$ for ($d\geq2$) with $p\in \N^*$ classes $M_1,\ldots,M_p$ of rescalings.
Then, passing to a subsequence, there exists a portrait ${\bf F}$, a sequence ${(f_n,y_n,z_n)_n\in\Rat_{{\bf F},X}}$ and a dynamical system between trees of spheres $(\F,\T^X)$ such that
\begin{enumerate}
\item $\displaystyle { f}_n\overset{\lhd}{\underset{\phi_n^Y,\phi_n^Z}\longrightarrow}{ \F}$ and
\item $\forall i\in[\!\![1,p]\!\!], \exists v_i\in\T^Y$, $M_i\sim(\phi^Y_{n,v_i})_n$.
\end{enumerate}
\end{theorem}

%

\begin{proof}
After passing to a subsequence, we can suppose that the number of critical values of the $f_n$ and the number of their preimages, and their respective multiplicities are constant.

Suppose that $\forall n\in\N, M_n=Id$.
Denote by $g$ the corresponding rescaling limit. The map $g$ has at least three periodic repelling cycles. Take one point, on each of these cycles, $x^1,x^2$ and $x^3$. As the cycles are repelling, they still exist on a neighborhood of $g$. 
We can take $x^i_n$ of fixed period $p_i\in\N$ such that $(x^i_n)\to x^i$.
 Let
\begin{enumerate}
\item  $X_n$ be the union of the cycles of $x^1_n,x^2_n$ and $x^3_n$ ;
\item $Z_n$ be the union of $X_n$ and the set of critical values of $f_n$ and 
\item $Y_n$ be $f_n^{-1}(Z_n)$.
\end{enumerate}

After passing to a subsequence, we can suppose that the cardinals of $X_n, Y_n$ and $Z_n$ don't depend on $n$. After changing the representative, we can suppose that $x_0^i=x^i$. Define $x_n:X_0\to\S$ by $x_n(x^i)=x_n^i$. Then, passing to a subsequence, we define  $y_n$ and $z_n$ to be such that the following diagram commutes :

\centerline{
$\xymatrix{
 Y_0 \ar[rrr]^{y_n}    \ar[dd] _{{ f_0}} &&&Y_n\subset\S  \ar[dd]^{f_n} \\
& \ar@{_{(}->}[lu] \ar@{_{(}->}[ld] X_0 \ar[r]^{x_n} &X_n \ar@{^{(}->}[ru] \ar@{^{(}->}[rd]&\\
      Z_0 \ar[rrr]_{z_n}  &&&Z_n\subset\S.
  }$}
  
It follows that the $(f_n,y_n,x_n)$ are dynamical systems between marked spheres of portrait given by the restriction of $f_0$, and its corresponding degree function (again after extracting).
From Theorem \ref{thmcomp00}, there exist dynamical systems between trees of spheres $(\F,\T^X)$ which are approximable by this sequence, so dynamically approximable by this sequence according to Lemma \ref{dynconvbis}.

Let $v$ be the vertex separating $x^1,x^2$ and $x^3$ in $T^Y$.
Using Lemma \ref{ext22}, that we will prove later, we can suppose that the vertex $v$ is not forgotten by $F^k$ (indeed, $v$ separate three elements of $z$ and from this lemma we can assume that they are in $T^X$ and then apply Lemma \ref{definiX}). Define $v':=F^k(v)$.

We are going to prove that :
\begin{enumerate}
\item $\phi_{n,v}$ converges to a Moebius transformation $M$,
\item v'=v,
\item $f_v^k$ and $f_n^k$ are equivalents.
\end{enumerate}

The first claim is clear according to the fact that $\phi_{n,v}\circ x_n(x^i)\to a_v(x^i)$. For the second claim we remark that
\[ \phi_{n,v}\circ\phi_{n,v'}^{-1}\circ (\phi_{n,v'}\circ f_n^k\circ \phi_{n,v}^{-1})= \phi_{n,v}\circ f_n^k\circ \phi_{n,v}^{-1}.\]
Indeed, the right side converges to $ g$ and the term between parenthesis converges to $f^k_v$. We deduce that $f_v^k=g$ so the third claim follows and as $\phi_{n,v}$ converges to a Moebius transformation $M$, we have proven that $(M_n)_n=(Id)_n\sim(\phi^Y_{n,v_i})_n$.

Suppose that $(M_n)_n$ is a rescaling of period $k$. 
Given that \[\displaystyle ({ f}_n\overset{\lhd}{\underset{\phi_n^Y,\phi_n^Z}\longrightarrow}{ \F} ) \implies  (M_n\circ { f}_n\circ M_n^{-1} \xrightarrow[\phi_n^Y\circ M_n^{-1},\phi_n^Z\circ M_n^{-1}]{\lhd}\F ),\]
we can consider that $M_n=Id$ and use the preceding  case.
If we have more rescaling limits, we can adapt this proof by marking three periodic cycles for each rescaling limit.
 \end{proof}

First we define the following.
\begin{definition}[Extension] Let $\tilde X$, $\tilde Y$ and $\tilde Z$ be finite sets containing at least three elements with $\tilde X\subset \tilde Y\cap\tilde Z$. 
We say that $(\F:\T^{ Y}\to\T^{ Z},\T^{ X})$ is an extension of 
$(\tilde\F:\T^{\tilde Y}\to\T^{\tilde Z},\T^{\tilde X})$ if these are two dynamical systems between trees of spheres and if \begin{enumerate}
\item $\T^\star\lhd\T^{\tilde\star}$, 
\item $\tilde F|_{IV^Y}= F|_{IV^Y}$ and
\item $(\deg \tilde F)|_Y=\deg F$. 
\end{enumerate}
\end{definition}

We will write $  (\F,\T^X)\lhd (\tilde \F,\T^{\tilde X}) $, and more generally we use the notation 
\[(f_n, y_n, z_n)_n\lhd(f_n,\tilde y_n,\tilde z_n)_n\]
 when for every $n\in\N$ we have $(f_n, y_n, z_n)\lhd (f_n,\tilde y_n,\tilde z_n)$ and all the $(f_n,\tilde y_n,\tilde z_n)$ have the same portrait.

Before proving Lemma \ref{ext22}, we prove the following lemma.

\begin{lemma}\label{lablablab}
If $(x_n)_n$ and $(y_n)_n$ are sequences of spheres marked respectively by $X$ and $Y$ such that $(x_n)_n\lhd(y_n)_n$ and $\displaystyle  {x}_n\underset{\phi^X_n}\longrightarrow { \T}^X$ then after passing to a subsequence, there exists a tree of spheres $\T^Y$ such that:
\begin{enumerate}
\item $\displaystyle  {y}_n\underset{\phi^Y_n}\longrightarrow { \T}^Y$,
\item $\T^X\lhd\T^Y$ and
\item $\forall v\in T^X,\phi^X_{n,v}=\phi^Y_{n,v}$.
\end{enumerate}
\end{lemma}

\begin{proof}
Using Theorem \ref{thmcomp} (cf Introduction), after passing to a subsequence, we define a tree $\check\T^Y$ and a sequence $\check\phi_{n,v}$ for all $v\in\check T^Y$ such that $\displaystyle  {y}_n\underset{\check\phi_n}\longrightarrow {\check \T}^Y$.

Then, for every vertex $v\in T^X$ separating three elements of $X$, we consider the vertex $\check v$ in $\check T^Y$ separating the same three elements and we want to replace the $\check v$ by the $v$ in $\check T^Y$ to define a new tree $\T^Y$, and the $\check\phi_{n,\check v}$ by the $\phi_{n,v}$ such that the lemma follows immediately. 
This is possible if, when two triples of elements of $X$ separate the same vertex in $T^X$, then they do the same in $\check T^Y$.

Consider two triples $t_i=(\chi_1,\chi_2,\chi_3)$ and $t_2=(\check \chi_1,\check \chi_2,\check \chi_3)$ that are separated by the same vertex $v$ in $T^X$, but not in $\check T^Y$. After changing the labelings, we can consider that $v_1$ separate $\chi_1,\chi_2,\chi_3$ in $T^Y$ with $i_{\check v}(\chi_1)=i_{\check v}(\chi_2)$, and that $v_2$ separates $\check \chi_1,\check \chi_2,\check \chi_3$ in $T^Y$ with $i_{ v}(\check \chi_1)=i_{ v}(\check \chi_2)$ as in Figure \ref{remfond2}.
Define the Moebius transformations $M_i:\S_v\to\S_{v_i}$ to be such that for $1\leq j\leq 3$, $M_1$ maps $i_v(\chi_j)$ to $i_{v_1}(\chi_j)$ and $M_1$ maps $i_v(\check \chi_j)$ to $i_{v_2}(\check \chi_j)$.

Then we have $M_2\circ M_1^{-1}=\lim( \phi_{n,v_2}\circ\phi_{n,v}^{-1})\circ(\phi_{n,v}\circ\phi_{n,v_1}^{-1})=\lim  \phi_{n,v_2}\circ\phi_{n,v_1}^{-1}$. This would contradict Lemma \ref{noncomp}.
 \end{proof}

 \begin{figure}
  \centerline{\includegraphics[width=11cm]{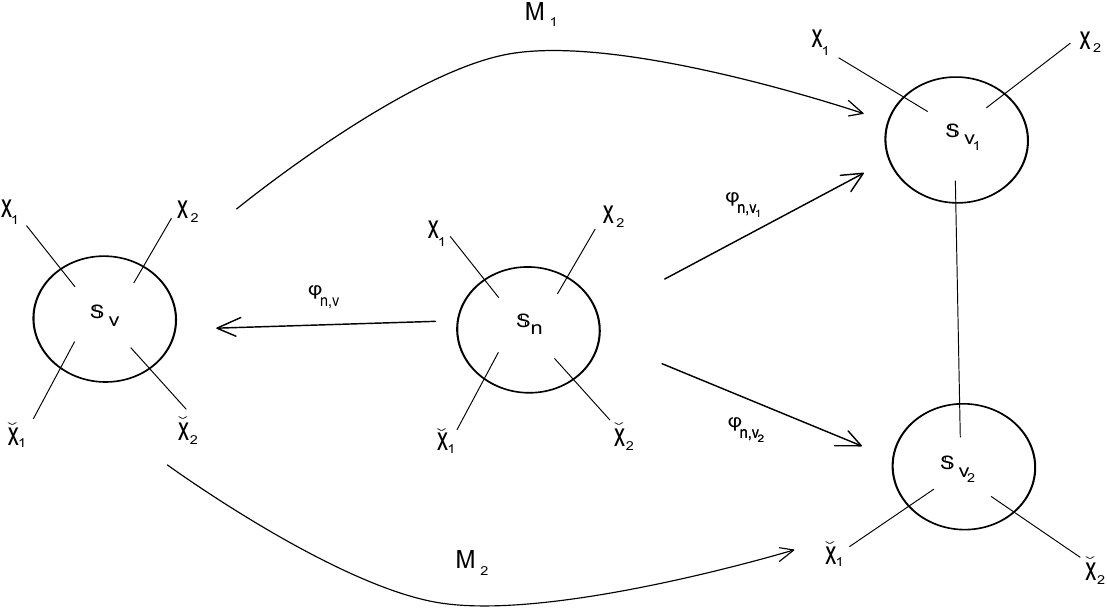}}
   \caption{Notations for the proof of Lemma \ref{lablablab}.}
\label{remfond2} \end{figure}

\begin{lemma}\label{ext22} Suppose that $\displaystyle { f}_n\overset{\lhd}{\longrightarrow}{ \F}$ and $z\in Z\setminus X$.
Then after passing to a subsequence there exist extensions $(f_n, y_n, z_n)_n\lhd(\tilde f_n,\tilde y_n,\tilde z_n)_n$ with $z\in \tilde X$ and $\forall n\in\N, \tilde x_n(z)=z_n(z)$ and $\tilde \F$ such that $\displaystyle {\tilde f}_n\overset{\lhd}{\longrightarrow}{\tilde \F}$  and 
\begin{enumerate}
\item $\T^X\lhd\T^{\tilde X}$, $\T^Y\lhd\T^{\tilde Y}$,
and $\T^X\lhd\T^{\tilde Z}$,
\item $\forall v\in IV^{ Y}, F(v)\in T^X\implies \tilde f_v=f_v$.
\end{enumerate}
\end{lemma}

\begin{proof}
 After passing to a subsequence we can assume that, either there exists $y\in Y$ such that $\forall n\in\N, z_n(z)=y_n(y)$, or $\forall n\in\N,z_n(z)\notin y_n(Y)$.

In the first case we define $\tilde X=X\cup\{z\}$. We define $\forall n\in \N, \tilde x_n(x)=x_n(x)$ for all $x\in X$ and $\tilde x_n(y)=y_n(y)$; we then have $(x_n)_n\lhd(\tilde x_n)_n$. Using Lemma \ref{lablablab}, we define a tree $\T^{\tilde X}$. Either $\T^{\tilde X}=\T^X\lhd\T^Y$, or $\T^{\tilde X}$ has exactly one more vertex then $\T^X$. In the latter case this vertex $v$ separate a $(y,x_1,x_2)$ with $x_1,x_2\in X$ and $(y,x_1,x_2)$ separates a unique vertex $v'$ in $\T^Y$. After replacing $v$ by $v'$ in $\T^{\tilde X}$, we have $\T^{\tilde X}\lhd\T^Y$ and the tree $\T^{\tilde X}$ still satisfies the conclusion of Lemma \ref{lablablab}.

We define $\tilde Y:=Y$ and $(\tilde y_n)_n:=(y_n)_n, \T^{\tilde Y}:=\T^Y$ and we then have $\T^Y\lhd\T^{\tilde Y}$.
 We identify $z$ and $y$ in $Z$, $(\tilde z_n)_n$ and $(z_n)_n$, $ \T^{\tilde Z}$ and $\T^Z$ after replacing the vertex separating $(y,x_1,x_2)$ in $\T^Z$ by $v'$. We then identify $\tilde \F$ and $\F$ according to the previous identifications and the result follows.
 
 In the second case (ie $\forall n\in\N,z_n(z)\notin y_n(Y)$) we define $\tilde X:=X\cup\{z\}$ and, using the same type of arguments as in the first case, we take the following steps.  
 
 $\star$ We set $\check Y=Y\cup\{(z_0(z))\}$. We construct a tree $\check \T^Y$ with $\T^Y\lhd\T^{\check Y}$ by extending $y_n$ to an injection $\check y_n$ with $\check y_n(z)= z_n(z)$, then a tree $\tilde \T^X$ with $\T^X\lhd\T^{ \tilde X}$ by extending $x_n$ to an injection $\tilde x_n$ with $\tilde x_n(z)= z_n(z)$. After a replacement of vertex on $\T^{\tilde X} $ we can suppose that $\T^{\tilde X}\lhd\T^{\check Y}$. 
 
  $\star$ We set $\tilde Z=Z\cup\{ f_0(y_0(z))\}$. We construct a tree $ \T^{ \tilde Z}$ with $\T^Z\lhd\T^{\check Z}$ by extending $z_n$ to an injection $\tilde z_n$ with $\tilde z_n(z)= f_n(y_n(z))$. After a replacement of vertex on $\T^{\tilde Z} $ we can suppose that $\T^{\tilde X}\lhd\T^{\tilde Z}$ (note that here we don't have necessarily $\T^{Z}\lhd\T^{\tilde Z}$).
 
 $\star$  We set $\tilde Y=Y\cup\{f_0^{-1}(f_0(y_0(z)))\}$, construct a tree $\T^{\tilde Y}$ with $\check \T^Y\lhd\T^{\tilde Y}$ by extending $\check y_n$ to an injection $\tilde y_n$ with $\tilde y_n(\tilde Y-Y)= f^{-1}_n(f_n(z))$ such that $\T^{\check Y}\lhd\T^{\tilde Y}$.
 
$\star$ Thus we have $\T^X\lhd\T^{\tilde X}$, $\T^Y\lhd\T^{\tilde Y}$,
and $\T^X\lhd\T^{\tilde Z}$.

 According to Theorem \ref{thmcomp00}, there exists a cover between trees of spheres \\${\tilde \F: \T^{\tilde Y}\to\T^{\tilde Z}}$ such that $(\tilde f_n,\tilde y_n,\tilde z_n)_n\to\tilde \F$.
Suppose that there exists a vertex  $v\in IV^Y$ such that $F(v)=v'\in T^X$ and $\tilde F(v)=v''\in T^{\tilde Z}$ with $v'\neq v''$. Then $v'\in T^{\tilde Z}$ because $ \T^X\lhd\T^{\tilde Z}$. Thus $\phi_{n,v'}\circ\phi_{n,v''}^{-1}$ converges uniformly outside a finite number of points to a constant. However, 
\[f_{v}=\lim \phi_{n,v'}\circ f_n\circ\phi^{-1}_{n,v}=\lim (\phi_{n,v'}\circ\phi^{-1}_{n,v''})\circ\phi_{n,v''}\circ f_n\circ\phi^{-1}_{n,v}\]
but $\phi_{n,v''}\circ f_n\circ\phi^{-1}_{n,v}$ converges uniformly to $\tilde f_v$ outside a finite set, therefor this is impossible and $v'= v''$. 
  \end{proof}


\subsection{{Theorem \ref{alpha}} and further considerations}

\begin{proof}[{Theorem \ref{alpha}}]  Take a sequence $(f_n)_n$ in $\Rat_d$ for $d\geq 2$ and suppose that it has strictly more then $p>2d-2$ dynamically independent  rescalings for which the associated rescaling limits are non post-critically finite. Then according to Theorem \ref{intro3}, passing to a subsequence, there exist a portrait ${\bf F}$, a sequence ${(f_n,y_n,z_n)_n\in\Rat_{{\bf F},X}}$ and a dynamical system between trees of spheres $(\F,\T^{ X})$ such that
\[\displaystyle { f}_n\overset{\lhd}{\underset{\phi_n^Y,\phi_n^Z}\longrightarrow}{ \F},\]
thus according to Theorem \ref{lienkiwi0} these classes of dynamically independent rescalings are associated to critic periodic cycles of spheres with a non post-critically finite associated cover. As ${\bf F}$ has degree $d$, because the $f_n$ lye in $\Rat_d$, this contradicts Theorem \ref{thmkiw1}.
 \end{proof}

We can see from the proof of Theorem \ref{intro3} that it is sufficient to mark some cycles to find the rescaling-limits but there is still an important question.

\begin{question}\label{question}How do we know which cycles we have to mark in order to find the rescaling-limits?\end{question}

In general this is not simple. For example, in a current work of  A. Epstein and C.L. Petersen (Limits of Polynomial-like Quadratic Rational Maps II, in preparation), the authors prove that we can have a non-trivial rescaling of any period in the case of degree $2$. 
There is another question that the reader should keep in mind. We defined dynamical systems between trees of spheres in a very general setting but the ones that lye to an interpretation in terms of rescaling limits are the one which are dynamically approximable by some sequence of dynamically marked rational maps. Hence we naturally arrive to the following:

\begin{question}\label{question2}Is every dynamical system between trees of spheres dynamically approximable by some sequence of dynamically marked rational maps?
\end{question}

The answer to this question is no, and a counterexample  is given in \cite[Figure 1]{A2}. This answer requires more technical results that are made explicit in \cite{A2}, where we give some necessary conditions for a dynamical system between trees of spheres to be approximable by some sequence of dynamically marked rational maps.





\end{document}